\documentclass[a4paper]{amsart}
\usepackage{amssymb, amsthm} 
\usepackage[pdftex]{hyperref}  
\usepackage{xypic}
\theoremstyle{plain}
\newtheorem{thm}{Theorem}
\newtheorem{corol}[thm]{Corollary}
\newtheorem{lemma}{Lemma}[section]
\newtheorem{prop}[lemma]{Proposition}

\theoremstyle{definition}
\newtheorem*{gapthm}{Gap Labeling Theorem}

\newtheorem*{claim}{Claim}

\theoremstyle{remark}
\newtheorem{rem}[lemma]{Remark}

\DeclareMathOperator{\tr}{tr}
\def\R{\mathbb{R}}
\def\C{\mathbb{C}}
\def\Z{\mathbb{Z}}
\def\N{\mathbb{N}}
\def\Q{\mathbb{Q}}

\def\P{\mathbb{P}}
\def\D{\mathbb{D}}

\def\H{\mathbb{H}}
\def\B{\mathbb{B}}
\def\eps{\varepsilon}
\newcommand{\cB}{\mathcal{B}}
\newcommand{\cF}{\mathcal{F}}

\newcommand{\calL}{\mathcal{L}}
\newcommand{\cO}{\mathcal{O}}\newcommand{\cP}{\mathcal{P}}
\newcommand{\cQ}{\mathcal{Q}}
\newcommand{\cS}{\mathcal{S}}\newcommand{\cT}{\mathcal{T}}
\newcommand{\cU}{\mathcal{U}}\newcommand{\cV}{\mathcal{V}}\newcommand{\cW}{\mathcal{W}}

\def\SL{\mathrm{SL}}

\def\SO{{\mathrm{SO}}}
\def\PSL{{\mathrm{PSL}}}

\def\supp{\operatorname{supp}}
\def\diam{\operatorname{diam}}
\def\id{{\operatorname{id}}}
\def\Id{{\operatorname{Id}}}  
\DeclareMathOperator{\interior}{int}
\renewcommand{\epsilon}{\varepsilon}
\renewcommand{\setminus}{\smallsetminus}
\renewcommand{\emptyset}{\varnothing}

\def\Leb{\mathrm{Leb}}
\def\Ruth{\mathit{Ruth}} 
\def\cUH{\mathcal{UH}}
\newcommand{\triplebar}[1]{|{\kern-.1 em}|{\kern-.1 em}|{#1}|{\kern-.1 em}|{\kern-.1 em}|}
\newcommand{\comm}[1]{}

\begin{document}

\title[Opening Gaps in the Spectrum of Schr\"odinger Operators]{Opening Gaps in the Spectrum of Strictly Ergodic Schr\"odinger Operators}

\begin{abstract}
We consider Schr\"odinger operators with dynamically defined
potentials arising from continuous sampling along orbits of
strictly ergodic transformations. The Gap Labeling Theorem states
that the possible gaps in the spectrum can be canonically labelled
by an at most countable set defined purely in terms of the
dynamics.  Which labels actually appear depends on the choice of
the sampling function; the missing labels are said to correspond
to collapsed gaps.  Here we show that for any collapsed gap, the
sampling function may be continuously deformed so that the gap
immediately opens.  As a corollary, we conclude that for generic
sampling functions, all gaps are open.  The proof is based on the
analysis of continuous $\SL(2,\R)$ cocycles, for which we obtain
dynamical results of independent interest.
\end{abstract}

\author[Avila]{Artur Avila$^{\flat,\sharp}$}
\urladdr{www.impa.br/$\sim$avila}
\email{artur@math.sunysb.edu}

\author[Bochi]{Jairo Bochi$^\ast$}
\urladdr{www.mat.puc-rio.br/$\sim$jairo}
\email{jairo@mat.puc-rio.br}

\author[Damanik]{David Damanik$^\ddagger$}
\urladdr{www.ruf.rice.edu/$\sim$dtd3} \email{damanik@rice.edu}

\thanks{
$^\flat$CNRS UMR 7599,
Laboratoire de Probabilit\'es et Mod\`eles al\'eatoires.
Universit\'e Pierre et Marie Curie--Bo\^\i te courrier 188.
75252--Paris Cedex 05, France. \\ 
\indent $^\sharp$IMPA, Estrada Dona Castorina 110, Jardim Bot\^anico, 22460-320 Rio de
Janeiro, Brazil.
\\
\indent $^\ast$Departamento de Matem\'atica, PUC-Rio,
Rua Marqu\^es de S\~ao Vicente, 225, 22453-900 Rio de Janeiro, RJ, Brazil.
\\
\indent $^\ddagger$Department of Mathematics, Rice University, Houston, TX 77005, USA
\\
\indent We are indebted to Jean Bellissard, Russell Johnson, and Barry Simon for useful comments and hints to the literature.
J.~B.\ would like to thank Rice University for its hospitality.
D.~D.\ would like to express his
gratitude to IMPA and {PUC--Rio} for
their hospitality and financial support through PRONEX--Dynamical Systems and
Faperj.
This research was partially conducted during the period A.~A.\ served as a Clay
Research Fellow.  J.~B.\ was partially supported by a grant from CNPq--Brazil.
D.~D.\ was supported in part by NSF~grants
DMS--0653720 and DMS--0800100.
We benefited from CNPq and Procad/CAPES support for traveling.
}

\date{March 12, 2009.}

\maketitle

\tableofcontents

\section{Introduction}

\subsection{Gap Opening}

Let $X$ be a compact metric space and let $f:X \to X$ be a strictly ergodic
homeomorphism (hence all orbits are dense and equidistributed with respect
to an invariant probability measure $\mu$).  We will also assume
that $f$ has a non-periodic finite-dimensional factor.\footnote{That is, there is
a homeomorphism $g:Y\to Y$, where $Y$ is a infinite compact subset of some Euclidean space $\R^d$,
and there is a onto continuous map $h:X \to Y$ such that $h \circ f = g \circ h$.}

Here we are interested in the one-dimensional discrete Schr\"odinger operator
$H=H_{f,v,x}$ whose potential is obtained by continuous sampling along an orbit of $f$.
Thus $H$ is a bounded self-adjoint operator on $\ell^2(\Z)$ given by
\begin{equation}\label{f.oper}
[H\psi](n) = \psi(n+1) + \psi(n-1) + V(n) \psi(n),
\end{equation}
where the potential $V$ is given by
$V(n) = v(f^n(x))$ for some $x \in X$ and some continuous {\it sampling
function} $v : X \to \R$.

Since $f$ is minimal, the spectrum of $H$ turns out to depend only
on $f$ and $v$, but not on $x$, and hence may be denoted by $\Sigma_{f,v}$; it is a perfect compact subset of $\R$.
Strict ergodicity implies that the
{\it integrated density of states} is well defined and only depends on $f$
and $v$: it is a
continuous non-decreasing surjective function $N_{f,v}:\R \to [0,1]$.
(See \S\ref{ss.review} for details.)
The basic relation between $\Sigma_{f,v}$ and $N_{f,v}$ is that $E \notin \Sigma_{f,v}$ if
and only if $N_{f,v}$ is constant in a neighborhood of $E$.  Hence, in each bounded
connected component of $\R \setminus \Sigma_{f,v}$ (called a {\it gap} of
$\Sigma_{f,v}$), $N_{f,v}$ assumes a constant value $\ell \in (0,1)$, called the {\it
label} of the gap.  A basic question is thus: which values of $\ell$ appear as
labels of gaps?

There is a well known general partial answer to this problem known
as the {\it Gap Labeling Theorem}.
Let $G \subset \R$ be the subgroup of all
$t$ such that there exist continuous maps $\phi:X \to \R$ and $\psi:X
\to \R/\Z$ with $t=\int \phi \, d\mu$ and
$\psi(f(x))-\psi(x)=\phi(x) \bmod{1}$.  Thus $G$ only depends on $f$
and it is called the {\it range of the Schwartzman asymptotic
cycle}\footnote{See \cite{S}, \cite[\S 5.2]{Parry}, \cite{A}; Theorem 4.2 from \cite{A} justifies this terminology.}.
Often $G$ is a non-discrete group: this happens for
instance
whenever $X$ has infinitely many connected components.

\begin{gapthm}

If $\ell$ is the label of some gap, then $\ell \in \mathcal{L}=G \cap
(0,1)$.

\end{gapthm}

Of course, the Gap Labeling Theorem does not guarantee that any
particular $\ell \in \mathcal{L}$ actually is  the label of some gap: it
is quite easy to see that the specific choice of the sampling
function is also important.
A missing label, that is, some $\ell \in \mathcal{L}$ which is
not the label of an actual gap, is usually said to correspond to
the {\it collapsed gap} $N_{f,v}^{-1}(\ell)$ (actual gaps are
sometimes called {\it open gaps} for this reason).

Describing which gaps are open in specific situations (and in
particular, whether the converse of the Gap Labeling Theorem
holds) is usually a very difficult problem.  The well known ``Dry
Ten Martini Problem'' (cf.~\cite{Si2}) is an instance of this: it asks one to show that
all gaps are open for the Almost Mathieu operator ($X=\R/\Z$,
$f(x)=x+\alpha$ for some $\alpha \in \R \setminus \Q$ and $v(x)=2
\lambda \cos 2 \pi x$ for some $\lambda>0$).  This problem is
still open, despite considerable attention (for partial progress,
see \cite {CEY}, \cite {P}, \cite {AJ}, and \cite {AJ2}).

Here we are concerned with the more general problem of whether
the dynamics, by itself, could {\it force} the closing of (many) gaps.
Questions in this direction arose early in the development of the theory: in
a 1982 paper, Belissard asks whether the mixing property would prevent the
set of labels of open gaps to be dense in $(0,1)$
(or equivalently, whether ``mixing avoids Cantor spectrum'',
\cite {Be0}, Open Problem~5).

We will show that the dynamics in fact does not obstruct the opening of
gaps, and indeed the converse of the Gap Labeling Theorem holds for
generic sampling functions.
This is obtained as a consequence of the
following result, which shows that any collapsed gap can be opened
by deformation:

\begin{thm}[Gap opening]\label{t.mainthm1}
For every $v \in C(X,\R)$ and any $\ell \in \mathcal{L}$,
there exists a continuous path $v_t \in C(X,\R)$,
$0 \leq t \leq 1$ with
$v_0 =v$ such that for every $t > 0$, the spectrum $\Sigma_{f,v_t}$
has an open gap with label $\ell$.
\end{thm}

Since the sampling functions for which the gap with label $\ell$
is open forms an open subset of $C(X,\R)$, we have the following:

\begin{corol}[Generic converse to the Gap Labeling Theorem]\label{c.gencgl}
For a generic $v \in C(X,\R)$, all gaps are open, that is, for every gap
label $\ell \in \mathcal{L}$,  the spectrum  $\Sigma_{f,v}$ has an open gap with label $\ell$.
\end{corol}

In particular, this gives, in our setting, an affirmative answer to another
question of Bellissard, who asked in \cite[Question~4]{Be1} whether Cantor
spectrum is generic whenever it is not precluded by the Gap
Labeling Theorem. (A specific mixing example falling in this category
is discussed below, answering his earlier question as well.)
On the other hand, as we will show in Appendix~\ref{a.appb}, the answer to Bellissard's
question becomes negative if one considers slightly more general base dynamics $f$.

\begin{rem}

We don't know whether for $v \in C(X,\R)$ given, it is possible to
find a continuous path $v_t \in C(X,\R)$, $0 \leq t \leq 1$ with
$v_0 =v$  and such that for every $t > 0$ all gaps of
$\Sigma_{f,v_t}$ are open.\footnote{The difficulty is that we
deduce Theorem~\ref{t.mainthm1} from Theorem~\ref{t.mainthm2}, and
that result does not hold with an extra parameter -- see
Remark~\ref{r.a4}.}

\end{rem}

\subsection{Examples}

One important class of examples (which includes the almost Mathieu operator)
is given by almost periodic Schr\"o\-dinger operators,
for which $X$ is an infinite compact Abelian group and
$f$ is a minimal translation: in this case $\mathcal{L}$ is a dense subset
of $(0,1)$. Two important subclasses are given by quasi-periodic
Schr\"odinger operators, where $X = \R^d / \Z^d$ for some $d \ge
1$, and by limit-periodic Schr\"odinger operators, where $X$ is a
product of finite cyclic groups.

If $\mathcal{L}$ is not dense in $(0,1)$, it is impossible for
$\Sigma_{f,v}$ to be nowhere dense. On the other hand, when $\mathcal{L}$
is dense in $(0,1)$, particularly in the almost periodic case,
much effort has been dedicated to establishing Cantor spectrum
(which means that a dense set of gaps are open).  For
limit-periodic Schr\"odinger operators, this issue was studied by
various authors in the 1980's (e.g., \cite{M}, \cite{AS} ,
\cite{C} , \cite{PT}) and, in fact, Corollary~\ref{c.gencgl} is
known in this particular case (see \cite{Si1} for an earlier
account in the periodic case). However, already for quasi-periodic
Schr\"odinger operators, the problem is much harder.

For the almost Mathieu operator, proving Cantor spectrum
corresponds to the Ten Martini Problem, which is now completely solved \cite
{AJ}, after several partial results (\cite {BS}, \cite {HS}, \cite {CEY},
\cite {L}, \cite {P}, \cite {AK}).
More generally, the case of non-constant
analytic sampling functions along translations of tori has been considered,
first by KAM-like methods and more recently by non-perturbative approaches.
For Diophantine translations on tori, in any dimension,
Eliasson \cite {E} has shown that for sufficiently small generic analytic
potentials, all gaps are open.
For general
non-constant analytic sampling functions $v$, Goldstein and Schlag \cite
{GS} showed that a dense set of gaps can be opened ``almost surely''
by changing $\alpha$, in the case
where the Lyapunov exponent is positive (which includes the case of
large $v$).

Previous to this work, we had shown in \cite {ABD}\footnote{Results related to \cite{ABD} can be found in
the papers \cite{CF}, \cite{FJ}, \cite{FJP}.} that generically,
a dense set of gaps are open, but only for a more restricted set of
dynamics (``generalized skew-shifts''): those are the dynamics which fiber
over an almost periodic dynamical system.  They also have dense $\mathcal{L}$.
An example in this class of particular interest is the skew-shift $(x,y) \mapsto
(x+\alpha,y+x)$.  However, generalized skew-shifts still form
a somewhat restricted class; for instance,
they are never {\it weakly mixing}.

In this respect, we would like to mention Fayad's analytic map $f$
on $X=\R^3/\Z^3$ from \cite{F} which is strictly ergodic, mixing
(hence not a generalized skew-shift), and has dense $\mathcal{L}$.\footnote{Details
are provided in Appendix~\ref{a.fayad}.}
Applying our results to this example yields, to the best of our
knowledge, the first example of Cantor spectrum for Schr\"odinger
operators arising from mixing dynamics.
Of course, an interesting example of a strictly ergodic
dynamical system is given by the time one map of the horocycle
flow for a compact hyperbolic surface. This is again mixing, but
in this case $\mathcal{L}$ is empty!

Other interesting examples of ``strictly ergodic dynamics'' which are
weakly mixing are given by typical interval exchange transformations \cite
{AF}.  They are not actually strictly ergodic
according to the usual definition since they are not homeomorphisms, but
this can be bypassed by going to a symbolic setting (blowing up procedure;
see~\cite {MMY}).

Of course, in the symbolic context, and more generally, among Cantor set
maps, there is a wealth of examples, and all have dense $\mathcal{L}$.  See also
\cite{BCL} for related examples of homeomorphisms of
manifolds.

\begin{rem}

In the analytic setting, it is far from clear whether one should
expect that the generic converse of the Gap Labeling Theorem, or
even generic Cantor spectrum, remains true for translations on
tori in higher (at least $2$) dimensions; compare \cite {CS}.
We do expect it to be true
in one dimension, which is supported by recent advances such as
\cite {AJ2} and \cite {GS}, but despite this progress, it is not yet
known whether the ``analytically generic'' spectrum is a Cantor
set (even restricting to Diophantine translations).

\end{rem}

\subsection{A Brief History of Gap Labeling}
Gap labeling theorems were first established in the early 1980's.
In 1982, Johnson and Moser \cite{JM} considered continuum
Schr\"odinger operators, $H = -\frac{d^2}{dx^2} + V$, with almost
periodic potential $V$. For $E \in \R$, if $u = u(x,E)$ is a
non-trivial real solution of the associated differential equation
$-u'' + V u = E u$, then $u'(x,E) + i u(x,E)$ never vanishes; this
allowed them to define the rotation number by
$$
\rho(E) = \lim_{x \to \infty} \frac{1}{x} \arg (u'(x,E) + i
u(x,E)).
$$
They showed that the limit exists, is independent of the solution
chosen, and defines a continuous non-decreasing function of $E$.
They also pointed out that due to Sturm oscillation theory, there
is a simple relation between the rotation number and the
integrated density of states. Since $V$ is almost periodic,
$M_V (\lambda) = \lim_{x \to \infty} \frac{1}{x} \int_0^x V(y)
e^{-i\lambda y} \, dy$ exists for every $\lambda \in \R$ and there
are at most countably many $\lambda$'s for which
$M_V (\lambda) \not = 0$. The set $\mathcal{M}(V)$ of finite integer
linear combinations of these $\lambda$'s is called the
\emph{frequency module} of $V$. Johnson and Moser then showed that
the function $\rho$ is constant precisely on the complement of the
spectrum $H$ and, moreover, if $E \in \R \setminus \sigma(H)$,
then $2 \rho(E) \in \mathcal{M}(V)$. Thus, one can label the gaps
of the spectrum by elements of the frequency module of $V$.

The analogue of the Johnson--Moser results for ergodic discrete
Schr\"odinger operators (and Jacobi matrices) can be found in the
1983 paper \cite{DS} by Delyon and Souillard. In essence, they
used Sturm oscillation theory and the easy existence proof for the
integrated density of states to define the rotation number in
their setting. The relation between the rotation number and the
integrated density of states is therefore built into this approach
and continues to hold trivially. Moreover, in the case of almost
periodic potential, the gaps in the spectrum may again be labeled
by elements of the frequency module of the potential.

Around the same time, Bellissard, Lima, and Testard began to
develop an approach to gap labeling based on $K$-theory of
$C^*$-algebras; see \cite{Be0} for the earliest account and
\cite{Be1,Be2,Be3,BLT} for extensive discussions. A major
advantage of this approach is that it is not restricted to one
space dimension. Let us discuss it in the one-dimensional discrete
setting relevant to our present study, but emphasize that the
results hold in greater generality. Given a compact space $X$ and
a homeomorphism $f$ of $X$, one can define an associated
$C^*$-algebra $C^*(X,f)$. An $f$-ergodic probability measure $\mu$
induces a trace $\tau$ on $C^*(X,f)$. Consider the $K_0$ group
associated with $C^*(X,f)$ and the induced map $\tau_* :
K_0(C^*(X,f)) \to \R$. The abstract gap labeling theorem now reads
as follows: Suppose $H$ is a self-adjoint operator affiliated to
$C^*(X,f)$ such that its integrated density of states exists with
respect to the $f$-ergodic measure $\mu$. Then the value of the
integrated density of states in a gap of the spectrum belongs to
$\tau_*(K_0(C^*(X,f))) \cap (0,1)$. This set of labels is not quite
easy to compute, in general. For the case of an irrational rotation of the circle,
this was accomplished using the Pimsner--Voiculescu cyclic six-term
exact sequence \cite{PV}. This and other computational aspects of
the theory are discussed in the references given above. In this context, we would
like to also mention the paper \cite{X} by Xia, which shows that the integers in the
quasi-periodic gap-labeling can be expressed as Chern characters.

The direct relation between taking the image under the trace
homomorphism of the $K_0$ group of $C^*(X,f)$ and taking the image
under Schwartzman's asymptotic cycle of the
first \v{C}ech cohomology group $\check H^1(Y_f,\Z)$,
where $Y_f$ is the suspension space of $(X,f)$, was
pointed out by Johnson \cite{J} in 1986; compare also Connes
\cite{Connes} and the discussions by Bellissard \cite{Be1} and
Kellendonk--Zois \cite{KZ}. On the other hand, as pointed out by
Bellissard and Scoppola in 1982 \cite{BSc}, a naive extension of
gap labeling by means of the frequency module to discontinuous
sampling function may not work. Indeed, they exhibited an example
where the potential arises by coding the rotation of the circle by
$\alpha$ irrational with respect to two intervals whose lengths
are irrational with respect to $\alpha$ and showed that the
integrated density of states takes values not belonging to $\Z +
\alpha \Z$.

In recent years, most of the literature on gap labeling has
focused on aperiodic tilings and Delone sets in higher dimensions.
While this is not immediately related to the subject of this
paper, we do want to point out that a recent highlight was
accomplished in three simultaneous and independent works by
Bellissard--Benedetti--Gambaudo \cite{BBG}, Benameur--Oyono--Oyono
\cite{BO}, and Kaminker--Putnam \cite{KP}. They proved a version of
the gap labeling theorem for aperiodic, repetitive tilings or
Delone sets of finite local complexity.

\subsection{Cocycle Dynamics} \label{ss.cocycle dynamics}

It is well established that spectral properties of one-dimensional Schr\"odinger operators
are intimately connected to a study of the solutions to the time-independent Schr\"odinger equation.
This equation in turn can be reformulated in terms of transfer matrices.
Since the potential is generated by sampling along the orbits of $f$ and
the dependence of the transfer matrices on the potential is local,
it turns out that these matrices have a so-called cocycle structure.
In particular, the dynamics of these $\SL(2,\R)$-valued cocycles describe the behavior of solutions
and hence are closely related to spectral information.
We will make this connection more explicit in the next subsection.
Here, we discuss basic notions related to $\SL(2,\R)$-valued cocycles
and present a relative of Theorem~\ref{t.mainthm1} in this context.

\medskip

Given a continuous map $A:X \to \SL(2,\R)$, we consider the
skew-product $X \times \SL(2,\R) \to X \times \SL(2,\R)$ given by
$(x,g) \mapsto (f(x), A(x) \cdot g)$. This map is called the
\emph{cocycle} $(f,A)$.\footnote{Since $f$ is fixed, we sometimes call $A$ a cocycle.}
For $n\in \Z$, $A^n$ is defined by
$(f,A)^n = (f^n, A^n)$.
The \emph{Lyapunov exponent} of the cocycle is the non-negative number:
$$
L(A) = \lim_{n \to \infty} \frac{1}{n}\log \|A^n(x)\| \, .
$$
(The limit exists for $\mu$-almost every $x$.)

We say that a cocycle $(f,A)$ is \emph{uniformly hyperbolic} if there
exist constants $c>0$, $\lambda>1$ such that $\|A^n(x)\| > c
\lambda^n$ for every $x\in X$ and $n>0$.\footnote{If $A$ is a real
$2 \times 2$ matrix then $\|A\| = \sup_{\|v\|\neq 0}
\|A(v)\|/\|v\|$, where $\|v\|$ is the Euclidean norm of $v \in
\R^2$.} This is equivalent to the usual hyperbolic splitting
condition: see \cite{Yoccoz}. Thus uniform hyperbolicity is
an open condition in $C(X, \SL(2,\R))$. Denote
$$
\cUH = \{ A \in C(X,\SL(2,\R) ; \; (f,A) \text{ is uniformly hyperbolic} \}.
$$

We say that two cocycles $(f,A)$ and $(f,\tilde A)$ are \emph{conjugate} if there exists
a \emph{conjugacy} $B \in C(X,\SL(2,\R))$ such that $\tilde{A}(x) = B(f(x))A(x)B(x)^{-1}$.
Uniform hyperbolicity is invariant under conjugation.

\medskip

Next suppose $A \in C(X,\SL(2,\R))$ is homotopic to a constant. Then
one can define a homeomorphism $F : X \times \R \to X \times \R$
so that $F(x,t) = (f(x), F_2(x,t))$ and $e^{2 \pi i F_2(x,t)}$ is
a positive real multiple of $A(x) \cdot e^{2 \pi i t}$, with the
usual identification of $\C$ with $\R^2$. We call $F_2(x,t)$ a
lift of $A$.  The lift is not uniquely defined, since for any
continuous function $\phi:X \to \Z$, $(x,t) \mapsto
F_2(x,t)+\phi(x)$ is also a lift, but this is the only possible
ambiguity.

Write $F^n_2(x,t)$ for the second component of $F^n(x,t)$.
Then the limit
$$
\rho_{f,A} = \lim_{n \to \infty} \frac{F^n_2(x,t) - t}{n}
$$
exists uniformly and is independent of $(x,t)$, it is called the
\emph{fibered rotation number} of the cocycle $(f,A)$; compare
\cite{H}. Due to the ambiguity in the choice of $F_2$, $\rho$ is
only defined up to addition of an element of the group
$G'=\{\int_X \phi \, d\mu: \phi \in C(X,\Z)\}$.  Clearly, $\Z
\subset G' \subset G$. Moreover, if $X$ is connected, then
$G'=\Z$, while if it is totally disconnected, then $G'=G$.

We will be interested in the variation of the fibered rotation number.
For any map $A_0:X \to \SL(2,\R)$ that is homotopic to a constant,
there is a neighborhood $\cV$ of it in $C(X,\SL(2,\R))$
such that for each $A \in \cV$, we can continuously select a lift $F_A : X \times \R \to X \times \R$.
Using those lifts, we define a map $A \in \cV \mapsto \rho_{f,A} \in \R$,
which is in fact continuous (see \cite{H}, \S5.7),
and is called a \emph{continuous determination of the fibered rotation number}.

\medskip

It turns out that (see~\cite {J})
\begin{equation}\label{f.uhrot}
A \in \cUH \; \Rightarrow \; 2\rho_{f,A} \in G.
\end{equation}
We will establish the following variant of Theorem~\ref{t.mainthm1} for $\SL(2,\R)$-cocycles:

\begin{thm}[Accessibility by uniformly hyperbolic cocycles]\label{t.mainthm2}
Suppose $A \in C(X,\SL(2,\R))$ is homotopic to a constant and
obeys $2\rho_{f,A} \in G$. Then there exists a continuous path
$A_t \in C(X,\SL(2,\R))$, $0 \leq t \leq 1$ with $A_0 = A$
and such that $A_t \in \cUH$ for every $t>0$.
{\rm (}In particular, $\rho_{f,A_t}$ is independent of $t$.{\rm )}
\end{thm}

In fact, as we will see in the next subsection,
Theorem~\ref{t.mainthm1} may be derived from Theorem~\ref{t.mainthm2}.

\subsection{Schr\"odinger Operators and Relation to Dynamics}\label{ss.review}

Let us now review in more detail the basic objects introduced in
the beginning and their relation to the dynamics of cocycles.

Consider the operator $H = H_{f,v,x}$ given by
\eqref{f.oper}. Since $f$ is minimal and $v$ is continuous, the
$x$-independence of the spectrum $\Sigma_{f,v}$ of $H$ can be
obtained from strong operator convergence modulo conjugation with
shifts. This common spectrum has a convenient description in terms
of the integrated density of states, which is defined as follows.
There is a probability measure $dN_{f,v}$ that obeys
$$
\int \langle \delta_0 , g(H) \delta_0 \rangle \, d\mu(x) = \int g(E) \, dN_{f,v}(E)
$$
for every bounded, measurable function $g$. Indeed, by the
spectral theorem, $dN_{f,v}$ is given by the $\mu$-average of the
$x$-dependent spectral measures associated with $H$ and
$\delta_0$.  The distribution function of $dN_{f,v}$ is called the
integrated density of states and denoted by $N_{f,v}$. By
construction, it is a non-decreasing function from $\R$ onto
$[0,1]$, which is constant in every connected component of $\R
\setminus \Sigma_{f,v}$. Moreover, it is not hard to show that it
is continuous and strictly increasing on $\Sigma_{f,v}$. That is,
the measure $dN_{f,v}$ is non-atomic and its topological support
is equal to $\Sigma_{f,v}$. There is an alternate way to define
$N_{f,v}(E)$ by restricting $H$ to finite intervals, counting the
eigenvalues below $E$, and taking a thermodynamic limit; compare
\cite{AS2}.
A third way is explained below.

\medskip

Spectral properties of general Schr\"odinger operators of the form
\eqref{f.oper} are most conveniently studied in terms of the
solutions to the one-parameter family of difference equations
\begin{equation}\label{eve}
u(n+1) + u(n-1) + V(n) u(n) = E u(n),
\end{equation}
where the energy $E$ belongs to $\R$.  Equivalently,
$$
\begin{pmatrix} u(n+1) \\ u(n) \end{pmatrix} = \begin{pmatrix} E - V(n) & - 1 \\ 1 & 0
\end{pmatrix} \begin{pmatrix} u(n) \\ u(n-1) \end{pmatrix}.
$$
Since in our situation, we have $V(n) = v(f^n x)$, we may consider the cocycle $(f,A_{E,v})$, where
$$
A_{E,v}(x) = \begin{pmatrix} E - v(x) & - 1 \\ 1 & 0
\end{pmatrix},
$$
and observe that $u$ solves \eqref{eve}
if and only if
$$
\begin{pmatrix} u(n) \\ u(n-1) \end{pmatrix} = A^n_{E,v} \begin{pmatrix} u(0) \\ u(-1).
\end{pmatrix}.
$$

Consequently, $(f,A)$ is called a \emph{Schr\"odinger cocycle} if $A$ takes its values in the set
\begin{equation}\label{e.def S}
\cS = \left\{ \begin{pmatrix} t & -1 \\ 1 & 0 \end{pmatrix} ; \; t \in \R \right\} \subset \SL(2,\R) \, .
\end{equation}

One of the fundamental results linking the spectral theory of $H$
and the dynamics of the family of Schr\"odinger cocycles
$(f,A_{E,v})$ is that $\Sigma_{f,v}$ consists of those energies~$E$
for which $(f,A_{E,v})$ is not uniformly hyperbolic:
\begin{equation}\label{jlresult}
\R \setminus \Sigma_{f,v} = \{ E \in \R : (f,A_{E,v}) \text{ is
uniformly hyperbolic} \};
\end{equation}
see Johnson \cite{J}.

The maps $A_{E,v}$ are homotopic to a constant and hence the
cocycle $(f,A_{E,v})$ has an associated rotation number which we
denote by $\rho_{f,v}(E)$.  For Schr\"odinger cocycles, there is a
canonical way to remove the ambiguity in the definition of
$\rho$,\footnote{Namely, one can uniquely choose the lift $F_2$ so
that $F_2(x,1/4)=1/2$ for every $x \in X$.} so that it can be
interpreted as an element of $[0,1/2]$ (and not of $\R/G'$ as in
the general case). There is a simple relation between $N_{f,v}$
and $\rho_{f,v}$ which follows from Sturm oscillation theory:
\begin{equation}\label{f.nandrho}
N_{f,v}(E) = 1 - 2\rho_{f,v}(E).
\end{equation}
It follows from \eqref{f.uhrot}, \eqref{jlresult}, and
\eqref{f.nandrho} that the gaps of $\Sigma_{f,v}$ (i.e., the
bounded connected components of $\R \setminus \Sigma_{f,v}$) may
be labeled by elements of $\mathcal{L} = G \cap (0,1)$ so that the label
of a gap corresponds to the constant value $N_{f,v}$ takes on it.

\medskip

To deduce Theorem~\ref{t.mainthm1} from Theorem~\ref{t.mainthm2}, we use the following result,
which is an improved version of \cite[Lemma~9]{ABD} and is proved in Appendix~\ref{a.2}.
(Recall that $\cS \subset \SL(2,\R)$ indicates the set of Schr\"odinger matrices,
as in \eqref{e.def S}.)

\begin{lemma}[Projection Lemma] \label{l.projection}
Let $f:X \to X$ be a minimal homeomorphism of a compact metric space with at least $4$ points, and let $A \in C^0(X,\cS)$.
Then there exist a neighborhood $\cW \subset C^0(X,\SL(2,\R))$ of $A$
and continuous maps
$$
\Phi=\Phi_A:\cW \to C^0(X,\cS) \quad \text{and} \quad
\Psi=\Psi_A:\cW \to C^0(X,\SL(2,\R))
$$
satisfying:
\begin{gather*}
\Psi(B)(f(x)) \cdot B(x) \cdot \left[\Psi(B)(x)\right]^{-1} = \Phi(B)(x), \\
\Phi(A)=A \text { and } \Psi(A) = \id.
\end{gather*}
\end{lemma}

In particular, an $\SL(2,\R)$-valued perturbation of a Schr\"odinger cocycle is
conjugate to an $\cS$-valued perturbation.

\begin{proof}[Proof of Theorem~\ref{t.mainthm1}]
Take a sampling function $v \in C(X,\R)$ and a label $\ell \in \mathcal{L} = G \cap (0,1)$.
Assume the gap of $\Sigma_{f,v}$ of label $\ell$ is collapsed,
that is, there is a unique $E_0 \in \R$ such that $N_{f,v}(E_0) = \ell$.
Let
$$
A(x) 
=  \begin{pmatrix} E_0 - v(x) & -1 \\ 1 & 0 \end{pmatrix} \, .
$$
By~\eqref{f.nandrho}, $\rho(A) = (1-\ell)/2$.
So we can apply Theorem~\ref{t.mainthm2} and find a continuous family of cocycles $A_t$, $t\in [0,1]$
with $A_0 = A$ and $A_t \in \cUH$ for $t>0$.

Apply Lemma~\ref{l.projection} and define $\tilde A_t = \Phi_A (A_t)$ for small $t$.
Since $\tilde A_t \in C(X, \cS)$, there is a continuous path $v_t \in C(X,\R)$
so that
$$
\tilde A_t (x) = \begin{pmatrix} E_0 - v_t(x) & -1 \\ 1 & 0 \end{pmatrix} \, .
$$
For $t>0$, we have $\tilde A_t \in \cUH$ and so, by \eqref{jlresult}, $E_0 \not \in \Sigma_{f, v_t}$.
Thus $N_{f,v_t}(E_0)$, which equals $1-2\rho(\tilde A_t) =\ell$, is a label of an open gap.
\end{proof}

\section{Statement of Further Results for $\SL(2,\R)$-Cocycles}\label{s.2}

This paper actually contains more refined results on cocycles,
and Theorem~\ref{t.mainthm2} is obtained as a corollary of one of them,
as we now explain.

\medskip

A cocycle $(f,A)$ is called \emph{bounded} if there exists $C>1$
such that $\|A^n(x)\| \le C$ for every $x\in X$ and $n\in \Z$. It
is known that a cocycle is bounded if and only if it is conjugate
to a cocycle of rotations, that is, to $(f,\tilde A)$ with $\tilde
A : X \to \SO(2, \R)$; see \cite{Cameron}, \cite{EJ}, \cite{Yoccoz}.

Our next main result is:

\begin{thm}[Accessibility by bounded cocycles]\label{t.bounded}
If $A \in C(X,\SL(2,\R)) \setminus \cUH$, then there
exists a continuous path $A_t \in C(X,\SL(2,\R))$, $t \in [0,1]$, such
that $A_0=A$ and $(f,A_t)$ is
conjugate to an $\SO(2,\R)$-valued cocycle for every $t \in (0,1]$.
Moreover, the conjugacy can be chosen to depend continuously on $t \in
(0,1]$.
\end{thm}

In particular, bounded cocycles are dense in the complement of $\cUH$.
In fact, we are able to prove this fact directly and
with weaker hypotheses on the dynamics:

\begin{thm}[Denseness of bounded cocycles]\label{t.addendum}
Assume that $f$ is a minimal homeomorphism of a compact metric space $X$
such that at least one of the following holds:
\begin{enumerate}
\item $X$ is finite-dimensional;
\item $f$ has at most countably many distinct ergodic invariant measures,
and $f$ has a non-periodic finite-dimensional factor.
\end{enumerate}
Then, the bounded cocycles form a dense subset of $C(X,\SL(2,\R)) \setminus \cUH$.
\end{thm}

\medskip

We now focus on the case of cocycles that are homotopic to a constant,
which form a subset of $C(X,\SL(2,\R))$ that we indicate by $C_0(X,\SL(2,\R))$.

Given $A \in C_0(X,\SL(2,\R))$, fix
a continuous determination $\rho$ of the fibered rotation number
on a neighborhood of it.
We say that the cocycle $A$ is \emph{locked} if
$\rho$ is constant on a neighborhood of $A$.
We say that $A$ is \emph{semi-locked} if it is not locked and
it is a point of local extremum for the function $\rho$.
Otherwise we say that $A$ is \emph{unlocked}.
This classification is obviously independent of the chosen determination of $\rho$.

In fact, a cocycle is locked if and only if it is uniformly
hyperbolic, and any semi-locked cocycle $A$ can be accessed by
uniformly hyperbolic cocycles. The proof of these statements and
some related facts may be found in Appendix~\ref{a.uhrn}. (Similar
results for ``non-linear cocycles'' over rotations can be found in
\cite{BJ}.)

For unlocked cocycles, we can strengthen the conclusion of
Theorem~\ref{t.bounded} considerably:

\begin{thm}[Accessibility by cocycles conjugate to a rotation]\label{t.reduce}
Let $A \in C_0(X,\SL(2,\R))$ be an unlocked cocycle. Let $\alpha
\in \R$ be such that $\alpha = \rho_{f,A} \mod G$. Then there
exist continuous paths $A_t \in C(X,\SL(2,\R))$, $t \in [0,1]$,
and $B_t \in C(X,\SL(2,\R))$, $t \in (0,1]$, such that $A_0=A$ and
$B_t(f(x))^{-1} A_t(x) B_t(x) = R_{2\pi\alpha}$ for $t \in (0,1]$.
\end{thm}

The assumption that the cocycle be unlocked is necessary:
see Remark~\ref{r.unlocked}.

It is useful to consider also a weaker notion of conjugacy:
We say that two cocycles $(f,A)$ and $(f,\tilde A)$ are \emph{projectively conjugate}
if there exists a \emph{projective conjugacy} $B \in C(X,\PSL(2,\R))$
such that $\tilde{A}(x) = B(f(x))A(x)B(x)^{-1}$ modulo sign.
The corresponding version of Theorem~\ref{t.reduce} is:

\begin{thm}[Accessibility by cocycles projectively conjugate to a rotation]\label{t.reduce psl}
Let $A \in C_0(X,\SL(2,\R))$ be an unlocked cocycle. Let $\alpha
\in \R$ be such that $2\alpha = 2\rho_{f,A} \mod G$. Then there
exist continuous paths $A_t \in C(X,\SL(2,\R))$, $t \in [0,1]$,
and $B_t \in C(X,\PSL(2,\R))$, $t \in (0,1]$, such that $A_0=A$
and $B_t(f(x))^{-1} A_t(x) B_t(x) = R_{2\pi\alpha}$ {\rm (}modulo
sign{\rm )} for $t \in (0,1]$.
\end{thm}

This result easily yields Theorem~\ref {t.mainthm2}:

\begin{proof}[Proof of Theorem~\ref{t.mainthm2}]
Suppose $A \in C(X,\SL(2,\R))$ is homotopic to a constant with
$2\rho_{f,A} \in G$; we want to access $A$ by a path $A_t$ in $\cUH$.
As mentioned above, if $A$ is locked, then $A$ is in $\cUH$ and thus we can take a constant path.
If $A$ is semi-locked, then a path of the form $A_t = R_{\eps t} A$ (for some small $\eps \neq 0$)
works; see Proposition~\ref{p.lockprop}.
Thus assume that $A$ is unlocked.

Apply Theorem~\ref{t.reduce psl} with $\alpha = 0$ and
find continuous paths $A_t \in C(X,\SL(2,\R))$, $t \in [0,1]$, and
$B_t \in C(X,\PSL(2,\R))$, $t \in (0,1]$, such that $A_0=A$ and
$A_t(x) = B_t(f(x)) B_t(x)^{-1}$ (modulo sign) for $t \in (0,1]$.
Let $\tau:(0,1] \to (1,\infty)$ be a continuous function. Take
$\tilde A_t \in C(X,\SL(2,\R))$ such that
$$
\tilde A_t(x) = B_t(f(x)) \begin{pmatrix} \tau(t) & 0 \\ 0 & 1/\tau(t) \end{pmatrix} B_t(x)^{-1}
\quad \text{(mod sign).}
$$
Choosing a function $\tau(t)$ that goes sufficiently fast to $1$
as $t\to 0$, we have $\lim_{t \to 0} \tilde A_t(x) = A(x)$. This
path has the desired properties.
\end{proof}

As for uniformly hyperbolic cocycles, they cannot, in general,
be approximated by cocycles that are conjugate to a constant (see Remark~3 in \cite{ABD}).
But for projective conjugacy there is no obstruction; indeed we show:

\begin{thm}[Accessibility by cocycles projectively conjugate to a hyperbolic matrix]
\label{t.reduce hyp}
Let $A \in \cUH$, and let $D \in \SL(2,\R)$ be such that $L(A)=L(D)$.
Then there exist continuous paths
$A_t \in C(X,\SL(2,\R))$,
$t \in [0,1]$, and $B_t \in C(X,\PSL(2,\R))$, $t \in (0,1]$, such that
$A_0=A$ and $B_t(f(x)) A_t(x) B_t(x)^{-1}=D$ {\rm (}in $\PSL(2,\R)${\rm )}
for $t \in (0,1]$.
\end{thm}

\begin{rem}\label{r.ruth}
Cocycles that are projectively conjugate to a constant are called
{\it reducible}.  Theorems \ref {t.reduce psl} and \ref {t.reduce
hyp} (together with the fact that the set of semi-locked cocycles
has empty interior) imply that reducibility is dense in
$C_0(X,\SL(2,\R))$. As for uniform hyperbolicity, it is dense in
$C_0(X,\SL(2,\R))$ if and only if $G$ is a non-discrete subgroup
of $\R$, by Theorem \ref {t.mainthm2} (and the basic
fact~\eqref{f.uhrot}). This statement cannot be significantly
improved; compare Theorem~\ref{t.furstenberg} in Appendix~\ref{a.appb}.

More generally, as in \cite{ABD}, we can let $\Ruth$ be the set of
all cocycles that are reducible up to homotopy
(or, equivalently, are projectively  conjugate to a cocycle  $C_0(X,\SL(2,\R))$;
see Lemma~8 in \cite{ABD}).
Any uniformly hyperbolic
cocycle is reducible up to homotopy, and in \cite{ABD} it is shown that
reducible uniformly hyperbolic cocycles are dense in $\Ruth$ in the
particular case where $f$ is a generalized skew-product.
Our results imply that in our more general context, reducibility is still dense
in $\Ruth$, and uniformly hyperbolic cocycles are dense in $\Ruth$ if and only if
$G$ is non-discrete.
\end{rem}

\begin{rem} \label{r.homotopy groups}
We also notice the following consequence of our methods: any
continuous cocycle $A$ that is not uniformly hyperbolic admits arbitrarily
small neighborhoods in $C(X,\SL(2,\R)) \setminus \cUH$ that are path
connected (and even have trivial homotopy groups in all dimensions).
See the comment after Proposition~\ref{p.extension2}.
\end{rem}

\section{Outline of the Paper and Discussion of the Methods}

Most of the paper is dedicated to building up the tools that will
be involved in the proof of cousin Theorems \ref {t.reduce} and
\ref {t.reduce psl}, with the other results being obtained on the
side. These tools can be split into two classes of results which
produce families of cocycles with certain features. The input in
either case is always some initial family of cocycles, but while
the first kind of result is only concerned with outputting a
suitable {\it perturbation} of the family, the second is primarily
concerned with the {\it extension} of the parameter space.

The perturbation arguments are developed in Sections \ref {s.almost inv} to
\ref {s.cohom}: they
refine the ideas introduced in \cite {ABD} to deal with individual
cocycles over generalized skew-shifts, both to address a
more general class of dynamical systems and to allow us to work in the
context of parametrized families.  Since many problems about
cocycles can be rephrased in terms of the existence of
invariant sections for an associated skew-product (for instance, conjugacy
to rotations is equivalent to the existence of an invariant section for the
associated skew-product arising from the disk action of $\SL(2,\R)$),
we first (Section \ref {s.almost inv}) discuss in a
more abstract setting the problem of existence of
{\it almost invariant sections}, satisfying the invariance condition except in
some specific set.  The actual estimates for the section depend on the
recurrence properties of the set and the existence of sets with good recurrence
properties is the topic of Section \ref {s.controlled}.
Those results are used
in Section \ref {s.disk} to show that a family of $\SL(2,\R)$-cocycles with small
Lyapunov exponents can be perturbed to become continuously conjugate to a
family of cocycles of rotations (Proposition \ref {p.disk adjust ext});
hence Theorem~\ref{t.addendum} follows as a particular case.
A similar result (Lemma \ref {l.line adjust ext}), regarding the
cohomological equation (which plays a role when conjugating a cocycle of
rotations to a constant), is obtained in Section~\ref {s.cohom},
and permits us to establish Theorem~\ref{t.reduce hyp}.

The extension arguments are developed in Sections \ref {s.cube 1} and \ref
{s.cube 2}: the
questions considered are rather novel, and their proofs need the
introduction of several ideas.  In Section \ref {s.cube 1}
we show that a family
of cocycles with zero Lyapunov
exponent, defined over the boundary of a cube, can be extended to a family
of cocycles with small
Lyapunov exponents, defined over the whole cube, provided there is no
topological obstruction (Lemma \ref {l.reduction}).  The proof is technically quite
involved, jumping from the continuous to the measurable category and back, and is
mostly independent from the rest of the paper (which never departs from the
continuous category): it is also the only place where ergodic theory plays a
significant role.  In Section \ref {s.cube 2}, Lemma \ref {l.reduction}
is used to
bootstrap a considerably more refined result
(Proposition \ref {p.extension2}): under the
weaker condition that the boundary values are not uniformly hyperbolic,
one gets the stronger conclusion that the extended family can be
continuously conjugated to cocycles of rotations in the interior of
the cube.
In particular, we obtain Theorem~\ref{t.bounded}.

In Section \ref {s.reduce},
everything is put together in the proof of Theorems \ref {t.reduce} and
\ref {t.reduce psl}.
We first show (Lemma \ref {l.step3}) that given $A$ unlocked, it is possible to
construct a two-parameter family $A_{s,t}$, $(s,t) \in [0,1]^2$ with
$A_{0,0}=A$, such that $\rho(A_{0,s})<\rho(A)<\rho(A_{s,0})$, and which is
continuously conjugate to rotations (except possibly at $(0,0)$).  This
almost implies that $A$ is accessible from the interior of the square
$[0,1]^2$ through a path with constant rotation number.  The actual
accessibility of $A$ by cocycles conjugate to a fixed rotation
is obtained with a little wiggling involving in particular the solution of
the cohomological equation.

The following diagram gives the logical dependence between the sections:
$$
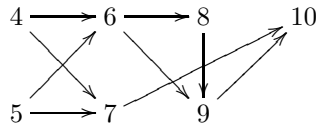
\xymatrix{
\ref{s.almost inv} \ar[r] \ar[rd] & \ref{s.disk}  \ar[r]\ar[rd] & \ref{s.cube 1}  \ar[d]  & \ref{s.reduce}\\
\ref{s.controlled} \ar[r] \ar[ru] & \ref{s.cohom} \ar[rru]      & \ref{s.cube 2} \ar[ru]  &  \\
}
$$

The paper also contains five appendices.
The first contains complementary results and examples concerning connectedness of
some sets of cocycles.
The second shows that complete gap labeling, and even just the Cantor structure of the spectrum, ceases to be generic when the assumption of unique ergodicity is dropped.
The third exposes some facts on locked / semi-locked / unlocked cocycles.
The fourth proves the projection lemma that was used to realize $\SL(2,\R)$ perturbations
by Schr\"odinger matrices.
Finally, in the last appendix we show how Fayad's results can be applied to obtain a mixing strictly ergodic example with a dense set of labels.

\section{Construction of Almost Invariant Sections} \label{s.almost inv}

{\it In this section and in the next section, $f:X \to X$ is an arbitrary
homeomorphism of a compact metric space.}  In fact, even if we are
ultimately interested in minimal dynamical systems, in order to
obtain ``parameterized'' results in the later sections, we will
apply the results obtained here to maps of the form $\hat
f(x,t)=(f(x),t)$ (the second coordinate representing the
parameter).

\medskip

Let $Y$ be a topological space and let
$F: X \times Y \to X \times Y$ be an invertible skew-product over $f$, that is,
a homeomorphism of the form $(x,y) \mapsto (f(x), F_x(y))$.  An {\it invariant
section for $F$} is a continuous map
$y:X \to Y$ such that $F_x(y(x))=y(f(x))$.
Throughout the paper,
a common theme will be the construction of {\it almost invariant sections}, which
satisfy $F_x(y(x))=y(f(x))$ in much of $X$.  We define the {\it support}
of an almost invariant section as the set of all $x$ such that
$F_x(y(x)) \neq y(f(x))$.
We will also be interested in estimating
``how large'' such almost invariant sections are.  In the abstract setting
considered here, this largeness will be measured by a continuous function
$M:Y \to [0, \infty)$.

We say that the pair $(Y, M)$ has the \emph{Tietze property} if
for any metric space $L$, every continuous function $y: K \to Y$
defined on a compact subset $K \subset L$ can be  extended to a continuous map $y: L \to Y$ such that
$$
\sup_{x \in L} M(y(x)) = \sup_{x \in K} M(y(x)) \, .
$$
The following are the examples of pairs $(Y,M)$ that we will use:
\begin{itemize}
    \item $Y=\R$ and $M(y) = |y|$;
    \item $Y$ is the unit disk $\D$ and $M(y) = \mathrm{d}(y,0)$, where $\mathrm{d}$ is the hyperbolic distance
    (see \S\ref{ss.disk adjustment}).
\end{itemize}
Using Tietze's Extension Theorem, we see that the Tietze property holds in these two examples.

We denote
$$
\triplebar{F}_M = \sup_{(x,y)\in X \times Y} \left| M(F_x(y)) - M(y) \right| \in [0,\infty].
$$
Notice that $\triplebar{F^{-1}}_M = \triplebar{F}_M$.

\medskip

To construct and estimate almost invariant sections whose support
is contained in the interior of
a compact set $K$, we will need only a few aspects of the recurrence
of the dynamics to $K$, which
we encode in the following notions.
We say that $K \subset X$ is \emph{$n$-good} if its first $n$
iterates are disjoint, that is, $f^i(K) \cap K=\emptyset$ for $1
\leq i \leq n-1$.
Let us say that $K \subset X$ is
\emph{$N$-spanning} if
the first $N$ iterates
cover $X$, that is, $\bigcup_{i=0}^{N-1} f^i(K)=X$.
We say that $K \subset X$ is \emph{$d$-mild} if
the orbit $\{f^n x : n \in \Z\}$ of each point in $X$ enters
$\partial K$ at most $d$ times.
The construction of sets with
appropriate recurrence properties, under suitable dynamical assumptions,
will be carried out in the next section.

The following abstract lemma
is the main result of this section.
(A simple particular case of it, where the dynamics fibers over an irrational rotation,
corresponds to Lemma~2 from \cite{ABD}.)

\begin{lemma}\label{l.almost invariant}
Let $f:X \to X$ be a homeomorphism of a compact metric space $X$,
and let $K \subset X$ be a $d$-mild, $n$-good, $N$-spanning set.
Let $\Lambda \subset X$ be a compact {\rm (}possibly empty{\rm )} $f$-invariant set.
Let $(Y,M)$ be a pair with Tietze property,
and let $F:X \times Y \to X \times Y$ be an invertible skew-product over $f$.
Let $y_0: X \to Y$ be an almost invariant section supported outside $\Lambda$.
Then there exists an almost invariant section $y:X \to Y$ supported in $\interior K$
with the following properties:
\begin{enumerate}
    \item $y$ equals $y_0$ on $\Lambda$;
    \item $\sup_{x \in X} M(y(x)) \le \sup_{x \in X} M(y_0(x))
+ d \max_{j \in [n, N]} \triplebar{F^j}_M$.
\end{enumerate}
\end{lemma}

\begin{proof}
We begin by using the set $K$ and the dynamics $f$ to decompose the space $X$.

For each $x \in X$, let
\begin{gather*}
\ell^+(x) = \min\{j \ge 0 : f^j(x) \in \interior K\}, \quad \text{and} \quad
\ell^-(x) = \min\{j >0 : f^{-j}(x) \in \interior K\}; \\
T(x) = \{j \in \Z : -\ell^-(x) < j < \ell^+(x) \}
\quad \text{and} \quad
T_B(x) = \{j \in T(x) :  f^j(x) \in \partial K\}.
\end{gather*}
Notice $T$ and $T_B$ are upper-semicontinuous.
Define
$N(x) = \# T_B(x)$ and $X^i = \{x \in X : N(x) \ge d-i\}$.
The sets $X^i$ are closed and
$$
X = X^d \supset X^{d-1} \supset \cdots \supset X^0 \supset X^{-1}
= \emptyset.
$$
Let also $Z^i = X^i \setminus X^{i-1} = \{x : N(x) = d-i\}$.

\begin{claim}
$\ell^+$ is locally constant on the set $Z^i$.
\end{claim}

\begin{proof}
We will show that $T_B$ and $T$ are locally constant on $Z^i$. Fix
$x\in Z^i$ and let $y \in Z^i$ be very close to $x$. We have
$T_B(y) \subset T_B(x)$. Since $y$ is also in $Z^i$, we must have
equality. Now, if $j \in T(x)$, then either $f^j(x) \in \partial
K$ or $f^j(x) \in X \setminus K$. In the former case, $f^j(y) \in
\partial K$ (because $T_B(y) = T_B(x)$), and in the latter,
$f^j(y)$ is also in $X\setminus K$ (because $y$ is close to $x$).
Since $f^{\ell^+(x)}(y)$, $f^{-\ell^-(x)}(y) \notin \interior K$,
we conclude that $T(y) =T(x)$.
\end{proof}

Let $Z^i_\ell = \{x \in Z^i : \ell^+(x) = \ell\}$; by the claim
this is a (relatively) open subset of $X^i$. Also notice
the following facts:
$$
\begin{gathered}
X^i = X^{i-1} \sqcup Z^i_0 \sqcup Z^i_1 \sqcup \cdots \sqcup Z^i_N \, , \qquad
\overline{Z^i_\ell} \setminus Z^i_\ell \subset X^{i-1}, \\
Z^i_0 = Z^i \cap \interior K, \qquad
f(Z^i_\ell) = Z^i_{\ell - 1} \text{ if $\ell>0$.}
\end{gathered}
$$

Now let $(Y,M)$, $F$, $\Lambda$, $y_0$ be as in the assumptions of the lemma.
We will describe a procedure to successively define $y$ on the set
$\Lambda \cup (X^0 \cap \interior K)$,
then on $\Lambda \cup X^0$,
then on $\Lambda \cup X^0 \cup (X^1 \cap \interior K)$,
then on $\Lambda \cup X^1$,
then on $\Lambda \cup X^1 \cup (X^2 \cap \interior K)$,
\dots,
then on $\Lambda \cup X^{d-1} \cup (X^d \cap \interior K)$,
and finally on $\Lambda \cup X^d = X$.

We begin by defining $y$ on $\Lambda \cup (X^0 \cap \interior K)$ as equal to $y_0$.

Let $i \in \{0,\ldots,d\}$.
Assume, by induction, that the map $y$ has already been continuously defined on the
set $\Lambda \cup X^{i-1} \cup (X^i \cap \interior K)$,
and has the two properties below:
\begin{gather}
F(x,y(x)) = (f(x), y(f(x))) \quad \text{for every $x$ in the domain and not in $\interior K$;}
\label{e.induction 1} \\
\sup_{x \in \Lambda \cup X^{i-1} \cup (X^i \cap \interior K)} M(y(x)) \le
\sup_{x \in X} M(y_0(x)) + iA, \label{e.induction 2}
\end{gather}
where $A = \max_{j \in [n, N]} \triplebar{F^j}_M$.

Now, for each $\ell =1, \ldots, N$, the map
$y$ is already defined on $f^\ell(Z^i_\ell) \subset X^i \cap \interior K$.
Naturally, we define $y$ on $Z^i_\ell$ by
$$
y(x) = \left(F^\ell_x\right)^{-1} \left( y(f^\ell(x)) \right).
$$
This defines $y$ on the set $\Lambda \cup X^i$
in such a way that \eqref{e.induction 1} holds,
and moreover
$$
\sup_{x \in \Lambda \cup X^i} M(y(x)) \le
\sup_{x \in \Lambda \cup X^{i-1} \cup (X^i \cap \interior K)} M(y(x))  + A \, .
$$

Let us check that the map $y$ so defined on $\Lambda \cup X^i$ is continuous.
Since $\Lambda$ and $X_i$ are closed sets and $y| \Lambda  = y_0|\Lambda$ is continuous,
we only need to check that $y|X^i$ is continuous.
Take a sequence $x_j$ in $X^i$ converging to some $x$ and let us show that $y(x_j) \to y(x)$.
If $x \in Z^i$, convergence follows from the claim above
and the continuity of $y|X^i \cap \interior K$.
Next assume $x \in X^{i-1}$.
Since $X^i = X^{i-1} \sqcup Z^i_0 \sqcup Z^i_1 \sqcup \cdots \sqcup Z^i_N$,
and $y|X^{i-1}$ is continuous, it suffices to consider the case where
the sequence $x_j$ is contained in some $Z^i_\ell$.
Then $f^\ell(x_j)$ and $f^\ell(x)$ all belong to $(X^i \cap \interior K) \cup X^{i-1}$,
where $y$ is continuous.
So, applying $F^{-\ell}$, we obtain $y(x) = \lim y(x_j)$, as desired.

At this point, $y$ is continuously defined on $\Lambda \cup X^i$.
If $i=d$, we are done.
Otherwise, we use the Tietze property to extend
continuously $y$ to $\Lambda \cup X^i \cup (X^{i+1} \cap \interior K)$
in a way such that
$$
\sup_{x \in \Lambda \cup X^i \cup (X^{i+1} \cap \interior K)} M(y(x))
=  \sup_{x \in \Lambda \cup X^i} M(y(x)) \, .
$$
Then the hypotheses \eqref{e.induction 1} and \eqref{e.induction
2} hold with $i+1$ in place of $i$; thus we can increment $i$ and
continue the procedure.
\end{proof}

\section{Construction of Sets With Controlled Return Times}  \label{s.controlled}

In the applications of Lemma~\ref{l.almost invariant},
it would be bad if $N$ was much bigger than $n$
(because then we wouldn't get a useful estimate for the $M$-size of the almost invariant section).
This issue is settled by means of the following result:

\begin{prop}\label{p.setskn}
Let $f:X \to X$ be a homeomorphism of a compact metric space
that admits a non-periodic minimal finite-dimensional factor.
Then there is $d \in \N$ such that for every $n \in \N$, there is a compact set
$K \subset X$ that is $n$-good, $(d+2)n-1$-spanning, and $d$-mild.
\end{prop}

We mention that in the case where $f$ fibers over an irrational
rotation, the proposition can be proven (with $d=1$) by a simple
explicit construction based on continued fractions -- see Figure~1
from \cite{ABD}.

\medskip

If $\tilde f:\tilde X \to \tilde X$ is a factor of $f:X \to X$
and $\tilde K \subset \tilde X$ is $d$-mild, $n$-good and $N$-spanning for
$\tilde f$, then $\pi^{-1}(\tilde K)$ is
$d$-mild, $n$-good and $N$-spanning for $\tilde f$, where
$\pi:X \to \tilde X$ is a continuous surjective map such that
$\pi \circ f=\tilde f \circ \pi$.
Thus it suffices to prove this statement for the minimal non-periodic
finite-dimensional factor of $f$.
For the remainder of
this section, it will therefore be assumed that $X$ is finite-dimensional and
infinite and that $f$ is minimal.

\begin{lemma}\label{l.mild}
There exists $d \in \N$ such that for every compact set $K \subset X$
and every neighborhood $U$ of $K$, there exists a compact
neighborhood $K' \subset U$ of $K$ that is $d$-mild.
\end{lemma}

\begin{proof}
This follows easily from Lemmas 4 and 5 in \cite{AB}.
\end{proof}

Let now $d$ given by Lemma \ref{l.mild} be fixed.

\begin{lemma} \label{l.n,d}
Let $K \subset X$ be an $n$-good, $N$-spanning compact set.
Then there exists a compact neighborhood $K'$ of $K$ that is $n$-good,
$N$-spanning and $d$-mild.
\end{lemma}

\begin{proof}
If $K'$ is a compact neighborhood of $K$ sufficiently close to
$K$, then it is still $n$-good and $N$-spanning. So the result follows from
Lemma~\ref{l.mild}.
\end{proof}

\begin{lemma}\label{l.tower}
For every $n \in \N$, there exists an $n$-good,
$(d+2) n-1$-spanning compact set $K \subset X$.
\end{lemma}

\begin{proof}
Let $n$ be fixed.
Let us start with a compact set $K$ with non-empty interior whose
first $n$ iterates are disjoint (a small compact neighborhood of
any point will do).  By minimality, $K$ is $n$-good and $N$-spanning
for some $N$.  Let us show that if $N \geq (d+2) n$, then there exists an
enlargement of $K$ that is $n$-good and $N-1$-spanning.

Assume $K$ is not $N-1$-spanning.
Let $K'$ be a compact
neighborhood of $K$ that is $n$-good, $N$-spanning
and $d$-mild, and also
not $N-1$-spanning.
Let $K_*$, resp.\ $K'_*$,
be the set of $x$ in $K$, resp.\ $K'$,
such that for each $i$ with $1\le i < N$, the point $f^i(x)$
does not belong to $K$, resp.\ $K'$.
Let $Y$ be the non-empty set $\overline{K'_*}$.
Notice that $Y$ is contained in $K_*$ and thus is $N$-good.
To simplify notation, replace $K$ with $K'$.

If $x \in Y$ and $f^j(x) \in K$ with $1 \leq j \leq N-1$, then
$f^j(x) \in \partial K$. For $x \in Y$, let $J(x)$ be the set of
all $j$ with $1 \leq j \leq N-1$ and $f^j(x) \notin K$. Let $I(x)$
be the set of all $i$ with $1 \leq i \leq d+1$ such that $\{in,
in+1, \ldots, (i+1)n -1 \} \subset J(x)$. Since $K$ is $d$-mild,
$\{1,\ldots, N-1\} \setminus J(x)$ has at most $d$ elements and
$I(x) \neq \emptyset$ for every $x \in Y$.

If for each $x \in Y$ we choose a non-empty subset $\tilde I(x)$ of $I(x)$,
then the formula
$$
K' = K \cup \bigcup_{x \in Y} \bigcup_{i \in I_\epsilon(x)} f^{i n}(x)
$$
defines an $n$-good and $N-1$-spanning set.
For $K'$ to be compact, we need to guarantee that
\begin{equation}\label{e.compactness}
\text{for any $i$, the set $\{x\in Y : \tilde I(x) \ni i \}$ is closed.}
\end{equation}

For $x \in Y$ and $\epsilon>0$, let
$B(x,\epsilon) = \{y\in Y : d(y,x)< \epsilon\}$.
Let
$J_\epsilon(x)=\bigcap_{y \in B(x,\epsilon)} J(y)$ and
$I_\epsilon(x)=\bigcap_{y \in B(x,\epsilon)} I(y)$.
For any $x \in Y$, there is $\epsilon(x)>0$ such that
$J_{\epsilon(x)}(x)=J(x)$, and hence $I_{\epsilon(x)}(x)=I(x)$.
Since $Y$ is compact, it can be covered by finitely many balls
$B(x_k, \epsilon(x_k))$.
Let $\epsilon$ be a Lebesgue number for this cover.
Then, for each $x \in Y$, there exists some $k$ such that
$B(x,\epsilon) \subset B(x_k, \epsilon(x_k))$ and therefore
$$
I_\epsilon(x) = \bigcap_{y \in B(x,\epsilon)} I(y) \supset \bigcap_{y \in B(x_k,\epsilon(x_k))} I(y)
= I_{\epsilon(x_k)}(x_k) = I(x_k) \neq \emptyset.
$$
Defining $\tilde I(x) = I_\epsilon(x)$, it is easy to see that \eqref{e.compactness}
is satisfied.
Thus there exist an $n$-good and $N-1$-spanning compact set, as we wanted to show.
\end{proof}

\begin{rem}
Motivated by this result, we pose the following question: If $f$
is a homeomorphism of a compact space (possibly of infinite
dimension) without periodic points, is it true that for every
$n\in \N$, there exists a compact set whose first $n$ iterates are
disjoint and such that finitely many iterates cover the whole
space? A related result is Theorem~3.1 in \cite{BC}.
\end{rem}

Proposition~\ref{p.setskn}
follows readily from the previous lemmas.

\section{From Slow Growth to Invariant Sections for the Disk Action} \label{s.disk}

\subsection{The Disk Action and the Adjustment Lemma}\label{ss.disk adjustment}

The group $\SL(2,\R)$ acts on the upper half-plane $\H =\{w\in
\C : {\mathrm{Im}\, z>0}\}$ as follows:
$$
A = \begin{pmatrix} a & b \\ c & d\end{pmatrix} \in \SL(2,\R) \ \Rightarrow \
A \cdot w = \frac{aw+b}{cw+d} \, .
$$
(In fact, the action factors through $\PSL(2,\R)$.)
We fix the following conformal equivalence between $\H$ and
the disk $\D = \{z \in \C : |z|<1\}$:
$$
w = \frac{-iz-i}{z-1} \in \H \ \leftrightarrow \ z =
\frac{w-i}{w+i} \in \D \, .
$$
Conjugating through this equivalence, we get an action of $\SL(2,\R)$
on $\D$, which we also denote as $(A,z) \mapsto A\cdot z$.

The disk is endowed with the Riemannian metric
$$
v \in T_z\D \ \Rightarrow \  \|v\|_z = \frac{2|v|}{1-|z|^2} \, .
$$
Let
$\mathrm{d}(\mathord{\cdot}, \mathord{\cdot})$ denote the induced distance function.
The group $\SL(2,\R)$ acts on $\D$ by isometries.

It can be shown that
\begin{equation}\label{e.norm}
\|A\| = e^{\mathrm{d}(A \cdot 0,0)/2} \quad \text{for all $A \in \SL(2,\R)$.}
\end{equation}
In particular, $A \in \SO(2,\R)$ iff $A \cdot 0 = 0$.

Let us recall Lemmas~5 and 6 from \cite{ABD}:

\begin{lemma}
\label{l.firstadjust}
There exists a continuous map $\Phi:\D \times \D \to \SL(2,\R)$ such that
$\Phi(p_1,p_2) \cdot p_1 = p_2$ and
$\|\Phi(p_1,p_2)- \Id\| \leq e^{\mathrm{d}(p_1,p_2)}-1$.
\end{lemma}

\begin{lemma}[Disk Adjustment]\label{l.adjust}
For every $n \geq 1$, there exists a continuous map $\Psi_n :
\SL(2,\R)^n \times \D^2 \to \SL(2,\R)^n$ such that if
$\Psi_n(A_1,\ldots,A_n,p,q) = (\tilde A_1,\ldots,\tilde A_n)$,
then
\begin{enumerate}
\item $\tilde A_n \cdots \tilde A_1 \cdot p = q$ and \item
$\|\tilde A_i A_i^{-1} - \Id\| \leq \exp \left( \tfrac{1}{2n}
\mathrm{d}(A_n \cdots A_1 \cdot p,q) \right)-1$ for $1 \leq i \leq
n$.
\end{enumerate}
\end{lemma}

\subsection{Construction of an Invariant Section} \label{ss.invariant}

Let $f:X \to X$ be a homeomorphism of a compact metric space.
If $A \in C(X,\SL(2,\R))$, the cocycle $(f,A)$, as defined in the introduction,
is a skew-product on $X \times \SL(2,\R)$.  The disk action of $\SL(2,\R)$ allows us
to consider also another skew-product denoted by
$F^\D_{f,A}:X \times \D \to X \times \D$ and given by $F^\D_{f,A}(x,z)=
(x,A(x) \cdot z)$.  The existence of an invariant section for $F^\D_{f,A}$ is
easily seen to be equivalent to $(f,A)$ being conjugate to a cocycle of rotations.
Indeed, if $\Phi$ is given
by Lemma~\ref{l.firstadjust} and $B(x)=\Phi(z(x),0)$, then $B(f(x)) A(x) B(x)^{-1}$
is a rotation for every $x$.

If $A \in C(X,\SL(2,\R))$
(or, more, generally, $A(x)$ is a $2\times 2$ real matrix depending continuously on $x \in X$),
we denote $\|A\|_\infty = \sup_{x \in X} \|A(x)\|$.

\begin{prop}\label{p.disk adjust ext}
Let $f:X \to X$ be a homeomorphism of a compact metric space that admits
a minimal non-periodic finite-dimensional factor, and let $\Lambda \subset X$ be a {\rm (}possibly empty{\rm )}
compact invariant set.
For every $C$, $\epsilon >0$, there
exists $\gamma > 0$ such that the following holds.
Suppose that
$A:X \to \SL(2,\R)$ is continuous,
$$
\|A\|_\infty < C
\quad \text{and} \quad
\lim_{n \to \infty} \sup_{x \in X} \frac1n \log \|A^n(x)\| < \gamma,
$$
and that $F^\D_{f|\Lambda,A|\Lambda}$ admits an invariant section $z$.
Then $A|\Lambda$ extends to a continuous map $\tilde A:X \to \SL(2,\R)$
such that $\|\tilde A - A\|_\infty < \varepsilon$ and $z$ extends to
an invariant section $\tilde z$ for $F^\D_{f,\tilde A}$.
\end{prop}

\begin{proof}
Let $d$ be given by Proposition~\ref{p.setskn}.
Let $\gamma > 0$ be such that
$e^{[1 + 2d(d+1)] \gamma} < 1 + \varepsilon/C$.
Let $n_0$ be such that $n \ge n_0$ implies
$\|A^n(x)\| \le e^{n\gamma}$ for every $x \in X$.

Let $M:\D \to [0,\infty)$ be the hyperbolic distance to $0$.
Fix some extension of $z$ to an element of $C(X , \D)$ and
set $C_1 = \sup_{x \in X} M(z(x))$.
Choose $n \in \N$ so that
$n > \max \{n_0, \gamma^{-1} \log C, C_1/\gamma\}$.
Let $K$  be the corresponding
$d$-mild, $n$-good, $N$-spanning
set given by Proposition~\ref{p.setskn},
where $N = (d+2)n - 1$.

We have
\begin{alignat*}{2}
\triplebar{(F^\D_{f,A})^j}_M
&=   \sup_{x \in X} \sup_{w \in \D}  \left| \mathrm{d}(A^j(x) \cdot w,0) - \mathrm{d}(w,0) \right|
&\quad &\text{(by definition)}\\
&\le \sup_{x \in X} \left| \mathrm{d}(A^j(x) \cdot 0, 0) \right|
&\quad &\text{(since $\mathrm{d}$ is $\SL(2,\R)$-invariant)}\\
&\le 2\sup_{x \in X} \log \|A^j(x)\|
&\quad &\text{(by~\eqref{e.norm})}\\
&\le 2j \gamma, &\quad &\text{provided $j \ge n > n_0$.}
\end{alignat*}
Applying Lemma~\ref{l.almost invariant}
we obtain an almost invariant section $\hat z$ for $F^\D_{f,A}$, which is supported
in $\interior K$ and satisfies
$$
\sup_{x \in X} M(\hat z(x)) \le \sup_{x \in X} M(z(x))
+ d \max_{n \le j \le N} \triplebar{(F^\D_{f,A})^j}_M.
$$
Thus we obtain
$M(\hat z(x)) \le C_1 + 2 d N \gamma$
for every $x \in X$.

Let $\Psi_n$ be given by Lemma~\ref{l.adjust} and put
\begin{multline*}
\big(\tilde A(x), \tilde A(f(x)), \ldots, \tilde A(f^{n-1}(x))\big) = \\
= \Psi_{n}\big(A(x), A(f(x)), \ldots, A(f^{n-1}(x)), \hat z(x), \hat z(f^n(x)) \big),
\end{multline*}
for each $x \in K$.
This defines $\tilde A$  on $\bigsqcup_{m=0}^{n-1} f^m(K)$.
We let $\tilde A=A$ on the rest of $X$.  Clearly $\tilde A$ is a
continuous extension of $A|\Lambda$.

For each $x \in K$, and $1 \le m \le n - 1$,
let $\tilde z(f^m(x)) = \tilde A^m (x) \cdot \hat z(x)$.
This defines $\tilde z$ on $\bigsqcup_{m=1}^{n-1} f^m(K)$.
We let $\tilde z=\hat z$ on the rest of $X$.  It is easy to see that
$\tilde z$ is a continuous extension of $z|\Lambda$
and that $\tilde z$ is an invariant section for $\tilde A$.

To complete the proof, we need to check that
$\tilde A$ and $A$ are $\| \cdot \|_\infty$-close.
We have
\begin{align*}
\mathrm{d} \big( A^{n}(x) \cdot \hat z(x), \hat z(f^{n}x) \big)
&\le \mathrm{d}\big( A^{n}(x) \cdot \hat z(x), 0 \big)  + \mathrm{d}\big( 0,\hat z(f^{n}x) \big)    \\
&\le \mathrm{d}(\hat z(x), 0) + 2 \log\|A^{n}(x)\| + \mathrm{d} (0, \hat z(f^{n}x))\\
&\le 2(C_1 + 2dN\gamma) + 2n \gamma \\
&\le [4 + 2d(d+2) ] n \gamma.
\end{align*}
Hence, using Lemma~\ref{l.adjust},
\begin{align*}
\|\tilde A - A\|_\infty &\le C \|\tilde A A^{-1} - \Id\|_\infty \\
& \le C \sup_{x \in X} \left[ \exp \left[\tfrac{1}{2n} \mathrm{d} \big( A^{n}(x) \cdot \hat z(x),
\hat z(f^{n}x) \big)\right] - 1 \right] \\
& \le C(e^{[2+ d(d+2)] \gamma} - 1)\\
& < \varepsilon. \qedhere
\end{align*}
\end{proof}

Let us obtain Theorem~\ref{t.addendum} from the proposition just proved:

\begin{proof}[Proof of Theorem~\ref{t.addendum}]
Assume that $f:X \to X$ is a minimal homeomorphism with a finite-dimensional
non-periodic factor.
We claim that if $X$ is finite dimensional, or there exist at most countably many
ergodic invariant measures for $f$, then a generic $A \in
C(X,\SL(2,\R)) \setminus \cUH$ satisfies
\begin{equation} \label {e.unif}
\lim_{n \to \infty} \sup_{x \in X} \frac {1} {n} \log \|A^n(x)\|=0.
\end{equation}
In the first case,
this is the main result of \cite{AB}.
In the second case, \cite{Bochi} shows that for each ergodic invariant measure $\mu$,
for a generic $A \in C(X,\SL(2,\R))$, the Lyapunov exponent relative to
$\mu$ vanishes, so taking the (at most countable) intersection, we conclude that
for a generic $A \in C(X,\SL(2,\R))$, the Lyapunov exponent vanishes
simultaneously for all ergodic invariant measures, and it
follows from Proposition~1 in~\cite{AB} that~\eqref{e.unif} holds.

Applying Proposition \ref {p.disk adjust ext} with $\Lambda=\emptyset$, we
conclude that any cocycle  $A \in C(X,\SL(2,\R)) \setminus \cUH$ can be approximated by some
$\tilde A$ that is conjugate to a cocycle of rotations.
\end{proof}

\section{Solving the Cohomological Equation} \label{s.cohom}

\def\new{\mathrm{new}}

\emph{From now on, we assume as usual that $f:X \to X$
is strictly ergodic with invariant probability measure $\mu$
and has a non-periodic finite-dimensional factor.}

\begin{lemma}\label{l.line adjust ext}
Let $T$ be a locally compact separable metric space and let $T^*
\subset T$ be a closed subset {\rm (}possibly empty{\rm )}. Let $\varphi \in
C^0 (T \times X, \R)$ and $\psi \in C( T^* \times X, \R)$ be such
that $\varphi(t,x) = \psi(t,f(x)) - \psi(t,x) + c(t)$ for $(t,x)
\in T^* \times X$, where $c(t) = \int \varphi(t,\mathord{\cdot})
\, d\mu$. Then, for every continuous function $\varepsilon:T \to
\R_+$, there exist $\tilde \varphi \in C(T \times X, \R)$,
coinciding with $\varphi$ on $T^* \times X$, and $\tilde \psi  \in
C(T \times X, \R)$, coinciding with $\psi$ on $T^* \times X$, such
that $|\tilde \varphi(t,x) - \varphi(t,x)| < \varepsilon(t)$ and
$\tilde \varphi(t,x) = \tilde \psi(t,f(x)) - \tilde \psi(t,x) +
c(t)$ for all $(t,x) \in T \times X$.
\end{lemma}

The proof of this proposition is somewhat similar to Proposition~\ref{p.disk adjust ext},
using the abstract Lemma~\ref{l.almost invariant}
and Proposition~\ref{p.setskn}.

\begin{proof}

Let us assume first that $T$ is compact.  In this case, we may replace the
function $\varepsilon$ by a positive number.
Let the functions $\varphi$ and $\psi$ be given.
Let
$S_j(t,x) = \sum_{m=0}^{j-1} \varphi (t, f^m(x))$ and $c(t) = \int \varphi(t,x) \, d\mu(x)$.
Using that $T$ is compact, it is easy to see that
$\frac1j S_j(t,x)$ converges to $c(t)$ uniformly over $(t,x) \in T \times X$.

Let $d$ be given by Proposition~\ref{p.setskn}. Choose $\delta >
0$ with $[2d(d+2) + 1] \delta < \varepsilon$, and choose $j_0$
such that $\left| \frac1j S_j(t,x) - c(t)\right|< \delta$,
uniformly for $(t,x)\in T \times X$, and $j \ge j_0$. Then choose
$n \ge j_0$ such that $\frac{2\|\psi\|_\infty}{n} + [2d(d+2) + 1]
\delta < \varepsilon$. (Here $\|\psi\|_\infty$ is meant to be $0$
in the case that $T^* = \emptyset$.) By
Proposition~\ref{p.setskn}, there is a compact set $K$ that is
$d$-mild, $n$-good, and $N$-spanning, where $N = (d+2)n - 1$.

Let $\hat{X} = T \times X$, and $\hat f: \hat{X}\ni (t,x) \mapsto (t, f(x))$.
Consider the skew-product on $\hat{X} \times \R$ over $\hat{f}$ given by
$$
F = F^\R_{\hat{f}, \varphi}: (t,x,w) \mapsto  (t, f(x), w + \varphi(t,x) - c(t)).
$$
Let $\Lambda = T^* \times X$. By assumption, the map $\psi \in
C(\Lambda, \R)$ is an invariant section of the skew product
$F|\Lambda \times \R$ over $\hat{f}|\Lambda$. Let $\psi_0 \in
C(\hat{X}, \R)$ be a continuous extension of $\psi$, so that
$\|\psi_0 \|_\infty = \|\psi\|_\infty$. (If $T^*= \emptyset$, take
$\psi_0 \equiv 0$.) Then $\psi_0$ is an almost invariant section
for the skew-product $F$ supported outside $\Lambda$. Applying
Lemma~\ref{l.almost invariant} with the pair $(Y,M) = (\R,
|\mathord{\cdot}|)$, we obtain an almost invariant section $\hat
\psi \in C(\hat{X}, \R)$ supported in $\interior K$ (i.e.,
$\varphi(t,x) = \hat \psi(t,f(x)) - \hat \psi(t,x) + c(t)$,
provided $x \not \in \interior K$) that coincides with $\psi_0$
(and hence with $\psi$) on $\Lambda$, and
\begin{align*}
\| \hat \psi \|_\infty
&\le \|\psi_0\|_\infty + d \max_{j \in [n, N]} \triplebar{F^j}_M \\
&=   \|\psi\|_\infty   + d \max_{j \in [n, N]} \sup_{(t,x)} \left| S_j(t,x) - j c(t) \right| \\
&<   \|\psi\|_\infty   + d N \delta .
\end{align*}

Let us define a function $\tilde \varphi:\hat{X} \to \R$.
For each $t\in T$, $x \in K$, and $0 \le m < n$, set
$$
\tilde\varphi (t,f^m(x)) =
\varphi(t,f^m(x)) + \frac{\hat \psi(t,f^{n}(x)) - \hat \psi(t,x) - S_{n}(t,x) + n c(t)}{n}.
$$
This defines $\tilde \varphi$ on $T \times \bigsqcup_{m=0}^{n-1} f^n(K)$;
let it equal $\varphi$ in the complement of this set.
Then $\tilde \varphi$ is continuous and
$$
\|\tilde \varphi - \varphi\|_\infty
\le \frac{2\|\psi\|_\infty}{n} + [2d(d+2) + 1] \delta
< \varepsilon.
$$
Notice that if $t \in T^*$, then $\hat\psi(t,x) = \psi_0(t,x) =
\psi(t,x)$ and thus $\tilde \varphi(t,x) = \varphi(t,x)$.

Let us define another function $\tilde \psi:\hat{X} \to \R$.
For each $t\in T$, $x \in K$, and $0 \le m < n$, set
$$
\tilde \psi (t,f^m (x)) = \hat \psi(t,x) + \sum_{j=0}^{m-1}[\tilde\varphi(t,f^j(x)) - c(t)].
$$
This defines $\tilde \psi$ on $T \times \bigsqcup_{m=0}^{n-1} f^n(K)$;
let it equal $\hat\psi$ in the complement of this set.
Then $\tilde \psi$ is continuous and
$\tilde \varphi(t,x) = \tilde \psi(t,f(x)) - \tilde \psi(t,x) +  c(t)$ for all $(t,x) \in T \times X$.
Also, $\tilde \psi(t,x) = \psi(t,x)$ when $t \in T^*$.
So the functions $\tilde \varphi$ and $\tilde \psi$ have all the desired properties.

This concludes the proof in the case where $T$ is compact. Let us
now consider the general case.  The hypotheses on $T$ imply that
there exists an exhaustion of $T$ in the sense that
$T=\bigcup_{i=1}^\infty T_i$ with $T_i$ compact and $T_i \subset
\interior T_{i+1}$.

We initially define $\tilde \varphi$ and $\tilde \psi$ on
$T_1 \times X$, by applying the compact case with the data $T_\new=T_1$,
$T^*_\new=T_1 \cap T^*$, $\varphi_\new=\varphi|T_1$, $\psi_\new=\psi|(T_1 \cap T^*)$,
$\varepsilon_\new=\varepsilon|T_1$: this yields functions $\tilde \varphi_\new,\tilde
\psi_\new:T_1 \times X \to \R$, which we take as the definition of
$\tilde \varphi$ and $\tilde \psi$ on $T_1 \times X$.  By construction,
all requirements are
satisfied when we restrict considerations to $T_1 \times X$.

Assume we have already defined
$\tilde \varphi$ and $\tilde \psi$
on $T_i \times X$
so that all requirements are satisfied when
considerations are restricted to $T_i \times X$.
Choose a continuous function $\varphi_i:T_{i+1} \times X \to \R$
that coincides with $\tilde \varphi$ on $T_i \times X$ and with
$\varphi$ on $(T^* \cap T_{i+1}) \times X$, such that $\int \varphi_i(t,x)
d\mu(x)=c(t)$ for every $t \in T_{i+1}$ and satisfying
$|\varphi_i(t,x)-\varphi(t,x)|<\varepsilon(t)$ for every $(t,x) \in T_{i+1}
\times X$.
Now apply the compact case again with data
$T_\new=T_{i+1}$, $T^*_\new=T_i \cup (T_{i+1} \cap T^*)$,
$\varphi_\new=\varphi_i$, $\psi_\new=\tilde \psi|(T_i \cup (T_{i+1} \cap T^*))$ and
$\varepsilon_\new(t)=\varepsilon(t)-\sup_{x \in X}
|\varphi_i(t,x)-\varphi(t,x)|$. This yields functions $\tilde
\varphi_\new,\tilde \psi_\new:T_{i+1} \times X \to \R$, which we take as the
definitions of $\tilde \varphi$ and $\tilde \psi$ on $T_{i+1} \times X$.
Notice that, by construction,
the new definitions extend the previous ones, and
satisfy all the
requirements when we restrict considerations to $T_{i+1} \times X$.

By induction, this procedure yields functions $\tilde
\varphi$, $\tilde \psi:T \times X \to \R$. These functions are continuous
since they locally coincide with a continuous function obtained in a finite
stage of the induction (here we use that $\bigcup \interior T_i=T$), and
satisfy all the other requirements over $T \times X$ (since those can be
also verified at a finite stage of the induction).
\end{proof}

Now we are able to give the:
\begin{proof}[Proof of Theorem~\ref{t.reduce hyp}]
Let $A\in \cUH$, and let $D\in \SL(2,\R)$ be such that $L(D)=L(A)$.
Let us indicate
$$
\Delta(r) = \begin{pmatrix} e^r & 0 \\ 0 & e^{-r} \end{pmatrix}.
$$
We can assume that $D = \Delta(c)$  for some $c \in \R$.
It is not difficult (see, e.g., the proof of Lemma~4 in \cite{ABD})
to show that there are continuous maps $B:X \to \PSL(2,\R)$ and $\varphi: X \to \R$
such that
$A(x) = B(fx)^{-1} \, \Delta(\varphi(x)) \, B(x)$ (in $\PSL(2,\R)$)
and $\int \varphi \, d\mu = c$.

Apply Lemma~\ref{l.line adjust ext} (with $T = (0,1]$ and $T^* = \emptyset$)
to find continuous functions $\tilde \varphi$, $\tilde \psi: (0,1] \times X \to \R$ such that
\begin{align*}
| &\tilde \varphi(t,x) - \varphi(x) | < \eps(t) \quad \text{where $\eps(t)\to 0$ as $t\to 0$,}\\
&\tilde \varphi(t,x) = \tilde \psi(t, fx) - \tilde \psi(t,x) + c.
\end{align*}
Define
$$
A_t(x) = B(fx)^{-1} \, \Delta (\varphi (t,x)) \, B(x)
\quad\text{and}\quad
B_t(x) =  \Delta(-\tilde \psi(t,x)) \, B(x).
$$
Then $A_t \to A$ as $t \to 0$ and $A_t(x) = B_t(fx) \, D \, B_t(x)^{-1}$,
as desired.
\end{proof}

\section{Extending Families of Cocycles: Small Lyapunov Exponents} \label{s.cube 1}

Let $T$ be a topological space.
We say that $A_t$, $t \in T$ is a \emph{continuous family of cocycles}
if $A_t \in C(X,\SL(2,\R))$ for every
$t \in T$ and the map $T \ni t \mapsto A_t \in C(X,\SL(2,\R))$ is continuous.
If $\cB$ is a subset of $C(X,\SL(2,\R))$ that contains $\{A_t : t \in T\}$
and such that the map $T \ni t \mapsto A_t \in \cB$ is homotopic to a constant,
then we say that that \emph{the family of cocycles $A_t$ is contractible in $\cB$.}
For example,
a continuous family of cocycles $A_t$, $t \in \partial [0,1]^p$ taking values in a set $\cB$
is contractible in $\cB$ if and only if there is an extended continuous family of cocycles $A_t$, $t \in [0,1]^p$
also taking values in $\cB$.

Given a continuous family of cocycles $A_t$, $t \in T$, we say that
$z_t$, $t \in T$ is a \emph{continuous family of sections} if
each $t \in T$ corresponds to an invariant section
$z_t \in C(X,\D)$ for the skew-product $F^\D_{f,A_t}$ (see \S\ref{ss.invariant}),
and the map $T \ni t \mapsto z_t \in C(X,\D)$ is continuous.

\medskip

The parameter spaces that we use next are the $p$-dimensional cube
$[0,1]^p$ and its boundary $\partial [0,1]^p$. The goal of this
section is to prove the following result:

\begin{lemma}\label{l.reduction}
Let $p \geq 1$ and let $\cB \subset C(X, \SL(2,\R))$ be a bounded open set.
Let $A_t$, $t \in \partial [0,1]^p$ be a continuous family of cocycles that is contractible in $\cB$,
and such that $L(A_t) = 0$ for all $t \in \partial [0,1]^p$.
Then, for any $\gamma > 0$, we can find a continuous
family of cocycles $\hat{A}_t$, $t\in [0,1]^p$
taking values in $\cB$ and extending the family $A_t$, $t \in \partial [0,1]^p$,
such that $L(\hat A_t) < \gamma$ for every $t \in [0,1]^p$.
\end{lemma}

The proof of Lemma~\ref{l.reduction} requires several preliminaries.
A measurable set $Y\subset X$ is called a \emph{coboundary}
if there is a measurable set $Z$ such that $Y = Z \vartriangle
f^{-1}(Z)$ $\mu$-mod~$0$.
Non-coboundaries are relatively easy to find:

\begin{lemma}[Corollary~3.5 from \cite{Knill}]\label{l.noncbdry}
If $Z\subset X$ is a positive measure set, then there is a
measurable set $Y \subset Z$ that is not a coboundary.
\end{lemma}

Our interest in non-coboundaries is due to the following fact,
already used in~\cite{Knill}:

\begin{lemma}\label{l.swap}
Let $A: X \to \SL(2,\R)$ be such that $\log\|A\|$ is $\mu$-integrable.
Assume that there exist two measurable maps $e^1$, $e^2: X \to \P^1$
such that
$$
e^1(x) \neq e^2(x), \quad
A(x)\cdot \{e^1(x), e^2(x)\}= \{e^1(fx), e^2(fx)\} \quad
\text{for $\mu$-almost every $x$.}
$$
If the set $Y = \{x \in X : A(x)\cdot e^1(x) = e^2(fx) \}$ is
\emph{not} a coboundary, then the Lyapunov exponent $L(A)$
vanishes.
\end{lemma}

\begin{proof}
Assume that $L(A)>0$, and consider the two Oseledets directions $e^+$,
$e^-: X \to \P^1$.
It follows from Oseledets's Theorem that
$\{e^1(x), e^2(x)\} = \{e^+(x), e^-(x)\}$ for almost every $x$.
Let $W$ be the set of $x$ where $e^1(x) = e^+(x)$.
Then $W \vartriangle f^{-1}(W)$ is precisely the set $Y$.
Hence it is a coboundary.
\end{proof}

If $Z \subset X$ is a measurable set of positive measure, the
\emph{first-return map} $f_Z: Z \to Z$ is defined as $f_Z(x) =
f^{r_Z(x)}(x)$, where $r_Z(x)$ is the least positive integer $n$
such that $f^n(x) \in Z$. The probability measure
$\frac{\mu}{\mu(Z)}$ is invariant and ergodic with respect
to~$f_Z$. If  $A: X \to \SL(2,\R)$ is such that $\log\|A\| \in
L^1(\mu)$, then we define the \emph{induced cocycle} $A_Z : Z \to
\SL(2,\R)$ by $A_Z(x) = A^{r_Z(x)}(x)$.

\begin{lemma}[cf.\ Lemma~2.2 from \cite{Knill}]\label{l.induced cocycle}
    The Lyapunov exponent of the cocycle $A_Z$ over the dynamics $f_Z$
    and with respect to the ergodic measure $\frac{\mu}{\mu(Z)}$
    is equal to $L(A)/\mu(Z)$.
\end{lemma}

A well-known theorem by Lusin states that a measurable function can be altered
in a set of arbitrarily small measure so that it becomes continuous.
We will need a parameterized version of this:

\begin{lemma}\label{l.param lusin}
Assume $T$ is a compact metric space. Let $\psi: T \times X \to
\R$ be a Borel measurable function such that for every $t \in T$,
there is a full measure subset of $X$ consisting of points $x$
with the property that the function $t' \mapsto \psi(t', x)$ is
continuous on a neighborhood of~$t$. Then for every $\beta>0$,
there exists a continuous function $\phi: T \times X \to \R$ such
that for every $t \in T$, the  points $x \in X$ where $\phi(t,x)
\neq \psi(t,x)$ form a set of measure less than~$\beta$.
\end{lemma}

\begin{proof}
For each $t \in T$ and $r>0$, let
$\bar{B}(t,r)$ be the closed ball of radius $r$ centered at~$t$,
and
$$
G(t,r) = \big\{ x\in X : \text{$t' \mapsto \psi (t', x)$ is continuous on $\bar B(t,r)$} \big\} \, .
$$
For every $t \in T$, there is $\rho(t)>0$ such that
$\mu(G(t,\rho(t)))>1-\beta/3$. Take a cover of $T$ by finitely
many balls $B_i = \bar{B}(t_i, \rho(t_i))$. Now, for each $i$, let
$K_i \subset G(t_i, \rho(t_i))$ be compact with $\mu(K_i)>1 -
2\beta/3$. Consider the mapping $\Psi_i : K_i \to C(B_i, \R)$
given by $t' \mapsto (x \in B_i \mapsto \psi(t',x))$. Take a
subset $K_i' \subset K_i$ with $\mu(K_i') > 1-\beta$ such that
$\Psi_i | K_i'$ is continuous. Let $L = \bigcup_i B_i \times
K_i'$. Then $L$ is compact, $\psi$ restricted to $L$ is
continuous, and for each $t$, the set of $x$ such that $(t,x)$
belongs to $L$ has measure at least $1 - \beta$. Let $\phi: T
\times X \to \R$ be a continuous extension of $\psi|L$, given by
Tietze's Extension Theorem. Then $\phi$ has the required
properties.
\end{proof}

We say that $A_t$, $t \in T$ is an \emph{$L^\infty$ family of cocycles}
if $A_t \in L^\infty(X,\SL(2,\R))$ for every $t \in T$,
the map $(t,x)\mapsto A_t(x)$ is measurable
and $\sup_t \|A_t\|_\infty$ is finite.

We say that $A_t$, $t \in T$ is an \emph{almost-continuous family
of cocycles} if it is an $L^\infty$ family, $T$ is a compact metric
space, and for every $t\in T$, there is a full measure set $G_t
\subset X$ such that for each $x \in G_t$, the mapping $t' \mapsto
A_{t'} (x)$ is continuous in a neighborhood of~$t$.

We define a \emph{weak metric} on the set $L^\infty(X, \SL(2,\R))$ as follows:
$$
d_\mathrm{w}(A,B) = \inf \big\{\beta >0 : \mu \{x\in X: \|A(x) - B(x)\| > \beta \} < \beta  \big\} \, .
$$

\begin{lemma}\label{l.semicont}
Assume $A_t$, $t \in T$ is an almost-continuous family of cocycles.
Let $\lambda = \sup_{t \in T} L(A_t)$, and fix $M > \sup_t \|A_t\|_\infty$.
Then:
\begin{enumerate}
    \item ${\displaystyle \lim_{n \to \infty} \sup_{t \in T} \frac{1}{n} \int_X \log\|A_t^n\| \, d\mu = \lambda}$.
    \item For every $\gamma>\lambda$,
    there exists $\beta>0$ with the following properties.
    Assume that $B_t$, $t\in T$ is an $L^\infty$ family of cocycles such that $\| B_t\|_{\infty } < M$
    and $d_\mathrm{w}(B_t, A_t) < \beta$ for every $t\in T$.
    Then $L(B_t) < \gamma$ for every $t$.
\end{enumerate}
\end{lemma}

\begin{proof}
The limit in part~(a) exists by subadditivity, and is obviously bigger or equal than $\lambda$.
Let us prove the reverse inequality.

Fix $\eta>0$ and take $t \in T$.
Let $m \in \N$ be such that $\frac{1}{m} \int \log \|A_t^m\| \, d\mu < \lambda+\eta$.
Since $A_t$ is an almost-continuous family of cocycles, so is $A^m_t$.
Hence for any $\eps>0$, there exists $\delta>0$ such that
$$
\mu \big\{x \in X : \|A^m_{t'}(x) - A^m_t(x)\| < \eps \text{ for all } t'\in B(t,\delta) \big\} > 1- \eps \, .
$$
Since the family is essentially bounded,
we can find $\delta = \delta(t)>0$ such that
$\frac{1}{m} \int \log \|A_{t'}^m\| \, d\mu < \lambda + 2\eta$ for every $t'\in B(t,\delta)$.
It follows from subadditivity that there exists $\bar m = \bar m(t) \in \N$ such that
$$
\frac{1}{n} \int \log \|A_{t'}^n\| \, d\mu < \lambda + 3\eta \quad \text{for all $t'\in B(t,\delta)$ and $n \ge \bar m$.}
$$

Now take a cover of $T$ by finitely many balls $B(t_i, \delta (t_i))$.
Let $n_0 = \max_i \bar{m}(t_i)$.
Then
\begin{equation}\label{e.part a}
\frac{1}{n} \int \log \|A_t^n\| \, d\mu < \lambda + 3\eta \quad \text{for every $t\in T$ and $n \ge n_0$.}
\end{equation}
Since $\eta>0$ is arbitrary, part~(a) of the Lemma follows.

Let $\eps>0$ be small.
There exists $\beta>0$ such that if $B_t$, $t\in T$ is an $L^\infty$ family of cocycles
with $d_\mathrm{w}(B_t, A_t) < \beta$ for each $t$,
then $d_\mathrm{w}(B_t^{n_0}, A_t^{n_0}) < \eps$.
Now, if in addition $\|B_t\|_\infty < M$
and $\eps$ was chosen small enough, it follows from~\eqref{e.part a} that
$$
\frac{1}{n_0} \int \log \|B_t^{n_0}\| \, d\mu < \lambda + 4\eta,
$$
which implies $L(B_t) < \lambda + 4\eta$ for every $t\in T$.
This proves part~(b).
\end{proof}

After these preliminaries, we can give the:

\begin{proof}[Proof of Lemma~\ref{l.reduction}]
Let us first explain informally the steps of the proof:
\begin{enumerate}
\item We select a bounded cocycle $D$ in $\cB$
and a small tower for $f$ along which some products of $D$ are rotations.

\item By assumption, the given family $A_t$ of cocycles can be extended to
a family defined in the whole cube, which we also denote by $A_t$.

\item In the crucial step,
we find an almost-continuous family of cocycles $B_t$ that has zero exponents for all values of the parameter and is uniformly very close to ${A}_t$ everywhere except in the tower.
This is done using Lemma~\ref{l.swap}.

\item We use Lemma~\ref{l.param lusin} to approximate $B_t$ by a continuous family of cocycles $\tilde A_t$;
by Lemma~\ref{l.semicont} we get $L(\tilde A_t)$ small for every $t$.

\item Since the original family of cocycles $A_t$, $t\in\partial [0,1]^p$ has zero exponents,
we can interpolate it with the $\tilde A_t$ family near the boundary of the cube
and obtain a continuous family $\hat{A}_t$, $t\in [0,1]^p$ that
coincides with the initial family in the boundary,
and (by semicontinuity) has small exponent in the whole cube.

\end{enumerate}
Now let us give the actual proof.

Since the set $\cB$ is bounded and $L(A_t)=0$ for every $t\in \partial [0,1]^p$,
by part~(b) of Lemma~\ref{l.semicont} there is $\beta>0$ such that
whenever $A_t'$, $t\in \partial [0,1]^p$ is an $L^\infty$ family of cocycles taking values in $\cB$
and with $d_\mathrm{w}(A_t', A_t) < \beta$, we have
$L(A_t')< \gamma$ for every $t\in \partial [0,1]^p$.

Since the family $A_t$, $t \in \partial [0,1]^p$ is contractible in $\cB$,
we can extend it to a continuous family $A_t$, $t \in [0,1]^p$ also taking values in $\cB$.

Let $\eps>0$ be very small. (Actual smallness requirements will emerge
several times along of the proof.)

Let $t_\mathrm{c} = (\tfrac{1}{2}, \ldots, \tfrac{1}{2})$ be the center of the cube.
We may assume that $A_{t_\mathrm{c}}=A_t$ for some $t \in \partial [0,1]^p$.
Using Theorem~\ref{t.addendum}, choose a cocycle
$D$ close to $ A_{t_\mathrm{c}}$ that has an invariant section $z : X \to \D$.
Making if necessary a small perturbation, we can assume that $z$
assumes the same value $z_0$ in an open set.
By conjugating everything with some fixed matrix (and hence changing $\cB$, $\gamma$, $\beta$ accordingly),
we can assume that $z_0 = 0$.
Let $N$ be a large integer.
Take a non-empty open set $Z \subset X$ such that its iterates in
times $0 = n_0 < n_1 <\cdots <n_N$
are wholly contained in $\{z=0\}$,
and moreover $Z \cap f^i(Z) = \emptyset$ for $0 < i \le n_N + 1$.
Notice that if $x \in Z$, then $D^{n_i}(x) \in \SO(2,\R)$ for
$0 \leq i \leq N$.
By a further perturbation (and reducing $Z$ if necessary), we can
assume that in addition to the previous properties,
$D|f^i(Z)$ is constant for $0 \leq i \leq n_N$,
and equal to the identity matrix for $i=n_N$.
Reducing $Z$ further, we assume that
the tower $\bigcup_{i=0}^{n_N - 1} f^i(Z)$
has measure less than $\beta/2$.
Let $W=f^{n_N}(Z)$ and $F=f^{n_N}: Z \to W$.
For $x \in W$, let $\ell(x)>0$ be
minimal with $f^{\ell(x)}(x) \in Z$, and let $G:W \to Z$ be given by
$G(x)=f^{\ell(x)}(x)$.  

Since the set $\cB$ is open, 
we can modify the family $ A_t$, $t\in [0,1]^p$ near $t=t_\mathrm{c}$
so that $ A_{t_{\mathrm{c}}}$ equals $D$.

Let $P_t :W \to \SL(2,\R)$ be given by $P_t(x) =  A^{\ell(x)}_t(x)$.

\begin{claim}
There exists an $L^\infty$ family $\tilde P_t$, $t\in [0,1]^p$ with the following properties:
\begin{enumerate}
\item $\|P_t^{-1} \tilde P_t - \Id \|_\infty < \epsilon$  for
every $t$. \item $(t,x) \mapsto \tilde P_t(x)$ is finite-valued
and takes values in $\SL(2,\R) \setminus \SO(2,\R)$. \item For
every $t$, there is a full measure subset of $W$ consisting of
points $x$ such that the map $t' \mapsto \tilde P_{t'}(x)$ is
constant in a neighborhood of $t$.
\end{enumerate}
\end{claim}

\begin{proof}
Since $f$ is minimal and $Z$ is open, $\ell(x)$ is bounded and therefore so is
the function $P_t(x)$.
Choose then a large finite set of $\SL(2,\R) \setminus \SO(2,\R)$
approximating its image.

Since $(X,\mu)$ is a non-atomic Lebesgue space, we can assume $X$ is
the unit interval and $\mu$ is Lebesgue measure.
Let $\cQ$ be a partition of $\R^{p}$ into small dyadic cubes.
Fix a vector $v_0 \in \R^{p}$ whose coordinates are independent mod~$\Q$.
For each $x \in [0,1]$, let
$\cQ(x)$ be the restriction of the partition $xv_0 + \cQ$ to the unit cube.
Then every $t \in [0,1]^{p}$ belongs to the interior (relative to
$[0,1]^{p}$)
of an element of the partition $\cQ(x)$
for all but finitely many points $x \in X$.
Take $\tilde P_t(x)$ to be constant in each element of $\cQ(x)$, and taking
values in the large finite set.
This proves the claim.
\end{proof}

For each $x \in W$,  let $e^s_t(x)$ and  $e^u_t(x)$ be respectively
the most contracted and the most expanded directions by $\tilde P_t(x)$;
these two directions are uniquely defined because the matrix is not a rotation,
and are orthogonal.
Next define $e^s_t(x)$ and $e^u_t(x)$ for $x \in Z = G(W)$ by
$$
e^*_t(x)=\tilde P_t(G^{-1}(x)) \cdot e^*_t(G^{-1}(x)) \, , \quad
* \in \{s,u\} \, .
$$
These two directions are also orthogonal.

There are only finitely many possible values for $e^u_t(x)$ and
$e^s_t(x)$; let $\cF \subset \P^1$ be the set of all of them.
Define a function
$\measuredangle:\cF \times \cF \to [0,\pi)$ such that
$R_{\measuredangle(x,y)} \cdot x=y$.

Apply Lemma~\ref{l.noncbdry} and choose some set of positive measure
$Y \subset Z$ that is not a coboundary.
Let
$$
\theta(t,x) =
\begin{cases}
\measuredangle(e^u_t(x), e^s_t(F(x))) &\quad \text{if $x \in Y$,} \\
\measuredangle(e^u_t(x), e^u_t(F(x))) &\quad \text{if $x \in Z \setminus Y$.}
\end{cases}
$$
Notice that for every $t$ and for every $x$ in a full measure subset of $X$ (depending on~$t$),
the function $t' \mapsto \theta(t',x)$ is constant in a neighborhood of $t$.

Define
$$
B_t (x)=
\begin{cases}
D(x) \cdot R_{\theta(t,f^{-n_i}x)/N}         &\text{if $x\in f^{n_i}(Z)$, $0\le i < N$,} \\
D(x)      &\text{if $x\in f^{j}(Z)$, $j \in \{0,1,\ldots, n_N\} \setminus \{n_0, n_1, \ldots, n_{N-1}\}$,} \\
 A_t(x) \cdot [P_t(x)]^{-1} \cdot \tilde P_t(x) &\text{if $x\in W$,} \\
 A_t(x)                                         &\text{otherwise.}
\end{cases}
$$
Let us list some properties of the family of cocycles $B_t$, $t \in [0,1]^p$.
First, it is almost-continuous.
Second,
\begin{equation}\label{e.all zero}
L(B_t) = 0 \quad \text{for every $t \in [0,1]^p$.}
\end{equation}
Indeed, the induced cocycle $(B_t)_Z(x)$
is given by
$$
C_t(x) = \tilde P_t(F(x)) R_{\theta(t,x)}, \quad x \in Z.
$$
(Of course, this cocycle  is regarded over the dynamics $G\circ F$, which is the first return map to $Z$.)
The pair of measurable directions $e^u_t$, $e^s_t$ is preserved by the
action of the cocycle;
moreover the directions are swapped precisely for the points in $Y$.
Thus Lemma~\ref{l.swap} gives that for each $t$,
the measurable cocycle  $(G \circ F,C_t)$
has zero Lyapunov exponent
(with respect  to the ergodic invariant measure $\frac {1} {\mu(Z)} \mu$).
Thus \eqref{e.all zero} follows from Lemma~\ref{l.induced cocycle}.

A third property of the family $B_t$
is that for each $t$, $B_t(x)$ is close to $ A_t(x)$
for every $x$ outside a set of small measure.
Let us be more precise.
Let $\tau(t,x)$ equal $t_c$ if $x$ is in the tower $\bigcup_{i=0}^{n_N} f^i(Z)$
(which has measure less than $\beta/2$), and $t$ otherwise.
Then we can write
$$
B_t(x) =  A_{\tau(t,x)} (x) \cdot E_t(x) \, ,
$$
where $E_t$, $t \in [0,1]^p$ is an almost continuous family of cocycles
such that $\|E_t(x) - \Id\|< \epsilon$ for all $t$ and $x$.
(Here we used that $N$ is large.)

Since $\tau(t,x)$ depends continuously on $t$,
we can apply for instance Lemma~\ref{l.param lusin} and find
a continuous function $\tilde \tau: [0,1]^p \times X \to [0,1]^p$ such that
$\mu \{x \in X : \tilde \tau(t,x) \neq \tau(t,x) \}$ is very small for every $t$.

Let $\B_\eps$ be the set of $M \in \SL(2,\R)$ such that $\|M-\Id\|< \eps$.
Since $\eps$ is small, this set is homeomorphic to $\R^3$.
Thus we can apply Lemma~\ref{l.param lusin}
and find a continuous map $(t,x) \in [0,1]^p \times X \mapsto \tilde E_t(x) \in \B_\eps$
such that
$\mu \{x \in X : \tilde E_t(x) \neq E_t(x) \}$ is uniformly small for every $t$.

Now define
$$
\tilde A_t(x) =  A_{\tilde \tau(t,x)} (x) \cdot \tilde E_t(x) \, .
$$
Then $\tilde A_t$ is a continuous family of cocycles
and (because $\eps$ is small) it takes values in $\cB$.
We can choose $\tilde \tau(t,x)$ and $\tilde E_t(x)$ so that
$\sup_t d_\mathrm{w}(\tilde A_t, B_t)$ is as small as desired.
Therefore, by \eqref{e.all zero} and part~(b) of Lemma~\ref{l.semicont}, we can assume that
$$
L(\tilde A_t) < \gamma \quad \text{for every $t \in [0,1]^p$.}
$$
With appropriate choices, we also have that for each $t \in [0,1]^p$, the set
$$
G_t = \big\{x \in X : \tilde \tau(t,x) = t \big\}
\quad \text{has $\mu(G_t) > 1- \tfrac{3}{4} \beta$.}
$$

The last step of the proof is to modify the family $\tilde A_t$
to make it match with $A_t$ in the boundary of the cube,
while keeping the exponent small.
Let $\pi: [0,1]^p \setminus \{t_\mathrm{c}\}\to \partial[0,1]^p$ be a retraction.
Let $S$ be a neighborhood of $\partial [0,1]^p$ inside $[0,1]^p$,
thin enough so that for every $t \in S$, $\pi(t)$ is close to $t$ and
hence $A_{\pi(t)}$ is close to $ A_t$.
Let $\hat E_t$, $t \in [0,1]^p$ be a continuous family taking values in $\B_\eps$,
such that $\hat{E}_t(x) = \tilde{E}_t(x)$ for $t$ outside $S$,
and $\hat{E}_t(x) = \Id$ for $t \in \partial [0,1]^p$.
Let $\hat{\tau}: [0,1]^p \times X \to [0,1]^p$ be continuous
such that $\hat{\tau}(t,x) = \tilde{\tau}(t,x)$ if $t\not \in S$,
$\hat{\tau}(t,x) = t$ if $t\in \partial [0,1]^p$,
and moreover $\hat \tau(t,x) = t$
for every $t \in [0,1]^p$ and $x \in G_t$.
Define
$$
\hat A_t(x) =  A_{\hat \tau(t,x)} (x) \cdot \hat E_t(x) \, .
$$
Then the family $\hat A_t$, $t\in [0,1]^p$ takes values in $\cB$ and
continuously extends $A_t$, $t\in \partial [0,1]^p$.
For $t \in [0,1]^p \setminus S$, we have $\hat{A}_t = \tilde{A}_t$ and therefore $L(\hat{A}_t) < \gamma$.
For $t \in S$ and $x \in G_t$,
we have
$\hat A_t(x) =  A_t (x) \cdot \hat E_t(x)$, which is close to $A_{\pi(t)} (x)$.
Since $\eps$ is small, we have $d_\mathrm{w}( \hat{A}_t, A_{\pi(t)} ) < \beta$.
By definition of $\beta$, this implies $L(\hat{A}_t) < \gamma$.
This completes the proof of Lemma~\ref{l.reduction}.
\end{proof}

\section{Extending Families of Cocycles: Boundedness} \label{s.cube 2}

In this section, we will strengthen Lemma~\ref{l.reduction},
obtaining the following two results:

\begin{prop}\label{p.extension1}
Let $p \geq 1$, and let $\cB \subset C(X, \SL(2,\R))$ be a bounded open set.
Let $A_t$, $t \in \partial [0,1]^p$ be a continuous family of cocycles that is contractible in $\cB$
and admits a continuous family of sections $z_t$, $t \in \partial [0,1]^p$.
Then $z_t$ can be extended to a  continuous family of sections
for an extended family of cocycles $A_t$, $t \in [0,1]^p$ that also takes values in $\cB$.
\end{prop}

\begin{prop}\label{p.extension2}
Let $p \geq 1$, and let $\cB \subset C(X, \SL(2,\R))$ be a bounded open set.
Let $A_t$, $t \in \partial [0,1]^p$ be a continuous family of cocycles that is contractible in $\cB$
and such that $A_t \not\in \cUH$ for every $t \in \partial [0,1]^p$.
Then there exists an extended continuous family of cocycles $A_t$, $t \in [0,1]^p$
also taking values in $\cB$ and such that
the family $A_t$, $t \in (0,1)^p$ admits a continuous family of sections.
\end{prop}

Notice that Theorem~\ref{t.bounded}
is an immediate consequence of this proposition with $p=1$.
It also follows that the complement of $\cUH$ is locally connected.
This shows the statement made in Remark~\ref{r.homotopy groups}.

\medskip

We will first prove Proposition~\ref{p.extension1}, and then use
it together with Theorem~\ref{t.addendum} to obtain
Proposition~\ref{p.extension2}. If $A \in C(X,\SL(2,\R))$ and
$\eps>0$, we write $\cB_\eps(A) = \{\tilde A : \|\tilde A -
A\|_\infty < \eps\}$.

\begin{proof}[Proof of Proposition~\ref{p.extension1}.]
Let $A_t$, $t\in \partial [0,1]^p$ be a continuous family of
cocycles admitting a continuous family of sections. Assume that
this family is contractible in a bounded open set $\cB \subset
C(X,\SL(2,\R))$, and thus can be extended to a $\cB$-valued family
$\bar{A}_t$, $t\in [0,1]^p$. Let $\eps>0$ be such that for each $t
\in [0,1]^p$, the ball with center $\bar{A}_t$ and radius $2\eps$,
indicated by $\cB_{2\eps}(\bar{A}_t)$, is contained in $\cB$. Let
$\gamma>0$ be given by Proposition~\ref{p.disk adjust ext}, where
in place of $f$ we take the homeomorphism of $\hat{X} = X \times
[0,1]^p$ given by $\hat{f}(x,t) = (f(x),t)$, and take $\Lambda =
X \times \partial [0,1]^p$, and $C > \sup_{B \in \cB}
\|B\|_\infty$. Next apply Lemma~\ref{l.reduction} to the family
$A_t$, $t\in \partial [0,1]^p$ and to the reduced set $\cB' =
\bigcup_{t\in [0,1]^p} \cB_\eps (\bar{A}_t)$. We obtain a family
$\hat{A}_t$, $t\in [0,1]^p$ with values in $\cB'$ such that
$L(\hat A_t) < \gamma$ for each $t$. Using compactness of
$[0,1]^p$ and subadditivity, it is not hard to see that
$$
\lim_{n \to \infty} \sup_{(x,t) \in X \times [0,1]^p} \frac1n \log \|\hat{A}^n_t(x)\| < \gamma.
$$
Thus we can apply Proposition~\ref{p.disk adjust ext} and find a
continuous family of cocycles $\tilde A_t$, $t\in [0,1]^p$ that
equals $\hat{A}_t$ (and hence $A_t$) when $t\in \partial [0,1]^p$,
such that $\| \tilde A_t - \hat A_t \|_\infty< \eps$ (and thus
$\tilde A_t \in \cB$) for each $t$, and that admits a continuous
family of sections $\tilde z_t$, $t \in [0,1]^p$ extending the
original $z_t$, $t\in \partial [0,1]^p$.
\end{proof}

\begin{lemma} \label{l.balls}
For every $C>0$, there is $\bar \eps>0$ such that for every $A \in
C(X,\SL(2,\R))$ with $\|A\|_\infty < C$ and every $\eps \in (0,
\bar \eps]$, the following holds: Every continuous family of
cocycles $A_t$, $t \in T$ taking values in $\cB_\eps(A)$ is
contractible in $\cB_{\eps}(A)$.
\end{lemma}

\begin{proof}
    Easy and left to the reader.
\end{proof}

Now we are able to prove the second main result of this section:

\begin{proof}[Proof of Proposition~\ref{p.extension2}]
Let $A_t$, $t \in \partial [0,1]^p$ be a family of cocycles that is contractible in $\cB$,
such that no $A_t$ is uniformly hyperbolic.

Let $V_n=[2^{-n}, 1-2^{-n}]^p$.
Take $n_0$ large (to be determined later).
Most of the proof will be devoted to the description of a procedure to extend the family of cocycles $A_t$ to the set $[0,1]^p \setminus \interior V_{n_0}$, which we call the \emph{shell}.

For each $n \ge n_0$, let $W_n = V_{n+1} \setminus \interior V_n$.
Cover the set $W_n$ by closed cubes $W_{n,i}$ of edge length $2^{-n-1}$ whose interiors are disjoint.
Let $\cP$ be the family of all those cubes $W_{n,i}$.
Consider the sets $K_0 \subset K_1 \subset \cdots \subset K_p$, where
$K_q$ is the union of the $q$-dimensional faces of cubes in $\cP$. 
We are going to define $A_t$ successively for $t$ in $K_0$, $K_1$, \dots, up to $K_p$,
and so obtain the extension to the shell.

During this proof, we take in $\R^p$ the box distance given by
$d((t_i), (t'_i)) = \max_i {|t_i - t'_i|}$. Fix a sequence
$\epsilon_1 > \epsilon_2 > \cdots$ converging to $0$ such that if
$\tau$, $\tau' \in \partial [0,1]^p$ are $2^{-n}$-close, then
$\|A_\tau - A_{\tau'}\|_\infty < \epsilon_n$.

Take $t\in K_0$ and let us define $A_t$.
First choose some $\tau = \tau(t) \in \partial [0,1]^p$ such that the distance from $t$ to $\tau$
is as small as possible; then that distance equals $2^{-n}$ for some $n \ge 2$.
Since $A_\tau$ is not uniformly hyperbolic,
by Theorem~\ref{t.addendum}
there is a cocycle $B$ with $\|B - A_\tau\| < \epsilon_n$ that admits a continuous section.
Set $A_t = B$.

The definition of $A_t$ proceeds by an inductive procedure.
Assume that $A_t$ is already defined for $t \in K_q$.
Fix $F$ to be a $(q+1)$-dimensional face in $K_{q+1}$,
and let $B =B(F)$ be its relative boundary, that is, the union of the $q$-dimensional faces in $K_q$ that intersect $F$.
Choose any point $t_0$ in $B$, and let $r = r(B)$ be the diameter of $\{A_t : t \in B\}$ in $C(X,\SL(2,\R))$.
If $r$ is sufficiently small, then, by Lemma~\ref{l.balls}
combined with Proposition~\ref{p.extension1}, we can extend the
continuous family of cocycles $A_t$, $t\in B$ (which takes values
in the ball $\cB_{2r}(A_0)$) and its continuous family of sections
to the $(q+1)$-dimensional cube $F$, so that the extended family
of cocycles takes values in $\cB_{4r}(A_0)$. The smallness of $r$
will be proven by induction, of course. In fact, we will prove
that there exist small numbers $r_n^{(q)}>0$ with $\lim_{n \to
\infty} r_n^{(q)} = 0$ such that if $F$ as above has diameter
$2^{-n}$, then $r(B) \le r_n^{(q)}$.

First consider $q=0$, so $F$ is an edge and $B=\{t_0, t_1\}$. If
$2^{-n}$ is the edge length, then $t_0$ and $t_1$ are both within
distance $2^{-n+1}$ of the boundary of the unit cube, hence
$d(t_i, \tau(t_i)) \le 2^{-n+1}$. It follows that $d(\tau(t_0),
\tau(t_1)) < 2^{-n+3}$ and so
\begin{align*}
r(B) =
\|A_{t_0} - A_{t_1}\|_\infty &\le \|A_{t_0} - A_{\tau(t_0)}\|_\infty + \|A_{\tau(t_0)} - A_{\tau(t_1)}\|_\infty +
\|A_{\tau(t_1)} - A_{t_1}\|_\infty \\
&\le \epsilon_{n-1} + \epsilon_{n-3} + \epsilon_{n-1} < 3 \epsilon_{n-3} \, .
\end{align*}
Thus we set $r_n^{(0)} = 3 \epsilon_{n-3}$.
These numbers will all be small for all $n \ge n_0$, provided $n_0$ is chosen large enough.
This proves the $q=0$ case.

Now let us explain how to find $r_n^{(q)}$ assuming $r_n^{(q-1)}$ is known.
Take a $(q+1)$-dimensional cell $F$;
its relative boundary $B = B(F)$ is composed of $2^{q+1}$ cells of dimension $q$.
In each of these cells, the variation of $A_t$ is at most $4 r_n^{(q-1)}$.
Hence a crude estimate gives $r(B) \le 2^{q+3} r_n^{(q-1)} =: r_n^{(q)}$.
Take $n_0$ large enough so that these numbers will be small for all $n\ge n_0$.

At this point we have defined $A_t$ for $t$ in the shell.
To check continuity in the boundary of the unit cube,
take a sequence $t_j$ in $(0,1)^p \setminus V_{n_0}$ converging to some $\tau \in \partial[0,1]^p$.
For each $j$, consider a cube of the family $\cP$ that contains $t_j$,
and let $t_j'$ be one of its vertices.
Then $\|A_{t_j} - A_{t_j'}\|_\infty \to 0$ as $j \to \infty$ due to the variation estimates.
On the other hand, $\tau(t_j') \to \tau$ and $\|A_{\tau(t_j')} - A_{t_j'}\|_\infty \to 0$,
therefore $A_{t_j} \to A_\tau$.

Now, taking a smaller shell if necessary, we can assume that all
the values of $A_t$ belong to the open set $\cB$. Finally, apply
Proposition~\ref{p.extension1} once more to extend continuously
the family of cocycles from the shell to the whole unit cube,
still taking values in $\cB$, and in such a way that in $(0,1)^p$,
there is a continuous family of sections. This concludes the proof
of Proposition~\ref{p.extension2}.
\end{proof}

\section{Proof of Theorems~\ref{t.reduce} and \ref{t.reduce psl}} \label{s.reduce}

Let $A \in C_0(X,\SL(2,\R)) \setminus \cUH$ be an unlocked
cocycle. Fix a continuous determination of the fibered rotation
number $\rho:\cB_{\epsilon_0}(A) \to \R$ on the ball of radius
$\epsilon_0$ and center $A$. As 
proved in Appendix~\ref{a.uhrn} (see
Proposition~\ref{p.lockprop}), we have
$$
\rho(R_{-\theta} A) < \rho(A) < \rho(R_{\theta} A) \quad \text{for every sufficiently small $\theta>0$.}
$$

The first step is to show the following:
\begin{lemma}\label{l.step1}
There is a continuous family of cocycles $A_t$, $t \in [-1,1]$
such that $A_0 = A$,
$$
\rho(A_{-t}) < \rho(A) < \rho(A_t)  \quad \text{for every $t \in (0, 1]$,}
$$
and no $A_t$ is uniformly hyperbolic.
\end{lemma}

\begin{proof}
Choose $\theta_n<0$ an increasing sequence close to $0$ and
converging to $0$, such that $2\rho(R_{\theta_n} A) \notin G$ for
each $n$. The sequence $\rho(R_{\theta_n} A)$ is also strictly
increasing (cf.~Proposition~\ref{p.lockprop}). So, by continuity
of $\rho$, we can find a sequence $\delta_n$ decreasing to $0$
such that
$$
\rho(R_{\theta_{n-1}} A') < \rho(R_{\theta_n}A'') < \rho(R_{\theta_{n+1}}A''') < \rho(A'''')
\text{ for all $A'$, $A''$, $A'''$, $A'''' \in \cB_{\delta_n}(A)$.}
$$
Define an open subset of $C(X,\SL(2,\R))$ by:
$$
\cT_n = \bigcup_{\theta \in [\theta_n, \theta_{n+1}]} R_\theta \cB_{\delta_n}(A).
$$

\begin{claim}
For each $B \in \cT_n$, there exists a unique
$\beta \in [-\theta_{n+1}+\theta_{n-1}, -\theta_n]$ such that
$\rho(R_\beta B) = \rho(R_{\theta_{n+1}} A)$.
Moreover, the function $B \mapsto \beta$ is continuous.
\end{claim}

\begin{proof}
Let $B \in \cT_n$, so $B = R_\alpha A'$ for some
$\alpha \in [\theta_n, \theta_{n+1}]$ and $A' \in \cB_{\delta_n}(A)$.
Then
$$
\rho(R_{-\theta_{n+1}+\theta_{n-1}} B) \le \rho(R_{\theta_{n-1}} A') <
\rho(R_{\theta_{n+1}} A)  <
\rho(A') \le \rho(R_{-\theta_n} B) \, .
$$
By the Intermediate Value Theorem, there is $\beta \in
[-\theta_{n+1}+\theta_{n-1}, -\theta_n]$ such that $\rho(R_\beta
B) = \rho(R_{\theta_{n+1}} A)$. Uniqueness is a consequence of
Proposition~\ref{p.lockprop} and the fact that
$2\rho(R_{\theta_{n+1}} A) \not\in G$. Continuity follows.
\end{proof}

For each $n$, consider the family of cocycles $R_\theta A$, $\theta \in \{\theta_{n+1}, \theta_n\}$.
It is obviously contractible in $\cT_n$.
Thus we can apply Proposition~\ref{p.extension2} to find
a continuous one-parameter family of cocycles
with values in $\cT_n \cap \cUH^c$,
joining $R_{\theta_n} A$ to $R_{\theta_{n+1}} A$.
By concatenation, we find a continuous
family of cocycles $A_t$, $t\in [\theta_1, 0]$
such that
\begin{gather*}
A_{\theta_n} = R_{\theta_n} A \, , \\
t \in [\theta_n, \theta_{n+1}] \ \Rightarrow \ A_t \in \cT_n  \cap \cUH^c \, .
\end{gather*}
Let
$$
t_n = \inf \big\{ t\in [\theta_n, \theta_{n+1}]: \rho(A_{n,t}) = \rho(R_{\theta_{n+1}} A)  \big\} \,.
$$
Now let us define a ``correction'' function $\beta: [\theta_1,0] \to \R$.
Define $\beta(t) = 0$ for $t \in [\theta_n, t_n]$.
For each $t \in [t_n,\theta_{n+1}]$, let $\beta(t)$ be given by the claim above
such that $\rho(R_{\beta(t)} A_{n,t}) = \rho(R_{\theta_{n+1}} A)$.
Finally, put $\beta(0) = 0$.
Then $\beta$ is well-defined and continuous.

Now consider the continuous family of cocycles $\tilde A_t = R_{\beta(t)} A_t$.
We have $\tilde A_t \in \cUH^c$ for each~$t$
(indeed, for $t \in [t_n,\theta_{n+1}]$, this follows from $2\rho (\tilde A_t) \not\in G$).
Also, in each interval $[\theta_n, \theta_{n+1}]$
we have $\rho(\tilde A_t) \le \rho(R_{\theta_{n+1}} A)$,
so $\rho(\tilde A_t) < \rho(A)$ for every $t < 0$.

Applying an entirely analogous procedure, we find a family $\tilde
A_t$ for small non-negative $t$, such that $A_0=A$ and
$\rho(\tilde A_t) > \rho(A)$ for $t>0$. By reparametrizing, we
find the desired family.
\end{proof}

Next we prove the following strengthening of Lemma~\ref{l.step1}:
\begin{lemma}\label{l.step2}
There is a continuous family of cocycles $A_t$, $t \in [-1, 1]$
such that $A_0 = A$,
$$
\rho(A_{-t}) < \rho(A) < \rho(A_t)  \quad \text{for every $t \in (0, 1]$,}
$$
and the family $A_t$, $t\in [-1,1] \setminus \{0\}$ admits a continuous
family of sections.
\end{lemma}

\begin{proof}
Let $A_t$, $t\in [-1,1]$ be the continuous family of cocycles given by the previous lemma.
We can assume it takes values in a small neighborhood of $A$.
Consider the continuous family $B_t$, $t \in \partial [0,1]^2$
defined by $B_{(t_1,t_2)} = A_{t_1}$.
Since these cocycles are not uniformly hyperbolic,
we can apply Proposition~\ref{p.extension2}
to find an extended family $B_t$, $t \in [0,1]^2$
such that the restricted family  $B_t$, $t \in (0,1)^2$
is continuously reducible to rotations.
Since $\rho(B_{(t_1,0)}) > \rho(A)$ for each $t_1 \in (0,1]$,
we can find a continuous function $h: (0,1] \to (0,1/2]$
such that $\rho(B_{(t_1/2,h(t_1))}) > \rho(A)$ for each $t_1$.
That is, $\tilde A_t = B_{(t/2,h(t))}$, $t\in (0,1]$ defines a continuous family of cocycles
that admits a continuous family of sections,
satisfies $\rho(\tilde A_t) > \rho(A)$ for each $t$,
and  $\lim_{t \to 0} \tilde A_t = A$.
In the same way we find the desired cocycles for negative parameter.
\end{proof}

In the next step we find a family depending on two parameters:
\begin{lemma}\label{l.step3}
There exists a family $A_t$, $t\in [0,1]^2$ such that
$A_{(0,0)} = A$,
$$
\rho\big(A_{(0,s)}\big) < \rho(A) < \rho\big(A_{(s,0)}\big) \quad
\text{for } s\in (0,1],
$$
and the restricted family $A_t$, $t\in [0,1]^2\setminus \{(0,0)\}$
admits a continuous family of sections.
\end{lemma}

\begin{proof}
Let $A'_t$, $t\in [-1,1]$ be given by Lemma~\ref{l.step2}, so
$\rho(A'_{-t}) < \rho(A) < \rho(A'_t)$ for $t>0$, and  $A'_t$,
$t\in [-1,1] \setminus \{0\}$ admits a continuous family of
sections $z'_t$.

By Proposition~\ref{p.extension1}, we can find a family
$A''_t$, $t\in [-1,1]$, admitting a continuous family of sections $z''_t$,
such that  $A''_{\pm 1} = A'_{\pm 1}$ and $z''_{\pm 1} = z'_{\pm 1}$.
Next define $A_t$ for $t \in \partial [0,1]^2$ by
$$
A_{(t_1,t_2)} =
\begin{cases}
A'_{t_1-t_2}  &\quad\text{if $t_1=0$ or $t_2=0$,} \\
A''_{t_1-t_2} &\quad\text{if $t_1=1$ or $t_2=1$.}
\end{cases}
$$
Define analogously the family of sections $z_t$, $t \in \partial [0,1]^2 \setminus \{(0,0)\}$.

Consider the set
$\Gamma = \bigcup_{n=1}^\infty \partial [0,1/n]^2$.
Using Proposition~\ref{p.extension1} (with $p=1$) infinitely many times,
find an extended continuous family of cocycles $A_t$, $t\in \Gamma$
such that the cocycles $A_t$, $t\in \Gamma \setminus \{(0,0)\}$ admit a
continuous family of sections that extends the previously defined $z_t$.

Now consider the regions
$[0,1/n]^2 \setminus (0,1/(n+1))^2$; each is homeomorphic to a square, and $A_t$, $B_t$ are defined in the boundary.
Applying Proposition~\ref{p.extension1}  (with $p=2$)  we extend $A_t$, $z_t$ to each of those holes.
This can be done so that $\lim_{t \to (0,0)}A_t = A$.
This gives the desired family of cocycles.
\end{proof}

Now we need a purely topological result:

\begin{lemma}\label{l.topology}
Let $\theta:[0,1]^2 \to \R$ be a continuous function.  Assume that
$\theta(0,0) = 0$ and $\theta(0,s) < 0 < \theta(s,0)$ for every
$0< s \leq 1$. Let $K=\theta^{-1}(0)$ and let $U \subset (0,1]^2$
be a neighborhood of $K \setminus \{(0,0)\}$.  Then there exists a
path $\gamma:(0,1] \to U$ such that $\lim_{s \to 0}
\gamma(s)=(0,0)$.
\end{lemma}

\begin{proof}
Let $\epsilon:(0,1] \to (0,1]$ be a continuous increasing function
such that for every $t \in K \setminus \{(0,0)\}$,
$B_{\epsilon(|t|)}(t) \cap [0,1]^2 \subset U$.

For every $t \in K \setminus \{(0,0)\}$, choose a closed dyadic
square $\tilde D_t$ containing $t$ of diameter at least
$\epsilon(|t|)/10$ and at most $\epsilon(|t|)/5$, and let $D_t$ be
the union of all closed dyadic squares of the same diameter that
intersect $\tilde D_t$.  Let $W=\bigcup_{t \in K \setminus \{0\}}
D_t$.

Let $I=\{t = (t_1 , t_2) \in [0,1]^2 ;\; t_1 + t_2 = 1/2\}$. Then
$I$ intersects $W$ on finitely many of its connected components,
which we denote by $S_i$.  We claim that one of them accumulates
at $(0,0)$. Indeed, if this is not the case, let $T$ be the union
of the boundaries of the $S_i$ intersected with $\Delta = \{t =
(t_1 , t_2) \in [0,1]^2 : t_1 + t_2 \le 1/2\}$. Then there
exists a path joining $(0,1/2)$ to $(1/2,0)$ that is contained
everywhere in $(I \setminus \bigcup S_i) \cup T$.  But the sign of
$\theta$ is constant along this path, a contradiction.

Let $S$ be such a connected component that accumulates at $(0,0)$.
Write $S$ as a countable union of dyadic squares,
so that any of them does not contain another.
Order such squares with the property that $D_i \cap
D_{i+1} \neq \emptyset$.  Refine this sequence so that no squares are
repeated.  Join the centers of the squares by straight segments.
This gives the desired path.
\end{proof}

Now we can finish the proof of Theorems~\ref{t.reduce} and \ref {t.reduce
psl}:

\begin{proof}
Consider the continuous family of cocycles $\tilde A_t$,
$t\in [0,1]^2$ given by Lemma~\ref{l.step3};
for $t \neq (0,0)$, one can write
$$
\tilde A_t(x) = \tilde B_t(f x) R_{2 \pi \theta(t,x)} \tilde B_t(x)^{-1} \, ,
$$
where $\tilde B_t \in C(X,\SL(2,\R))$, $t \in
[0,1]^2\setminus\{(0,0)\}$ is a continuous family of conjugacies
and $\theta : [0,1]^2 \times X \to \R$ is continuous, with $\int
\theta(0,s,x) \, d\mu(x) < \int \theta(0,0,x) \, d\mu(x) <
\theta(s,0,x)$ for every $0< s \leq 1$. Here, we write $\int
\theta(0,0,x) \, d\mu(x)$ for $\lim_{t \to (0,0)} \int \theta(t,x)
\, d\mu(x)$. By definition of the fibered rotation number, $\int
\theta(0,0,x) \, d\mu(x) = \rho(A) \mod G'$
(where $G'$ is a subgroup of $G$, see \S\ref{ss.cocycle dynamics}).

Let $g=\alpha-\int \theta(0,0,x) \, d\mu(x)$, in the case of Theorem \ref
{t.reduce}, and let $g=2 \alpha-2 \int \theta(0,0,x) \, d\mu(x)$, in the case
of Theorem \ref {t.reduce psl}.  In either case, $g \in G$, so
we may write $g=\int \phi \, d\mu$ for some continuous function
$\phi:X \to \R$ such that $\phi=\psi \circ f-\psi$ for
some continuous $\psi:X \to \R/\Z$.  In the case of Theorem \ref {t.reduce},
we may thus replace $\tilde B_t(x)$ by $\tilde
B_t(x) R_{2 \pi \psi(x)}$ and $\theta(t,x)$ by $\theta(t,x)+\phi(t,x)$, that is, we may
actually assume that $g=0$.  In the case of Theorem \ref {t.reduce psl},
we may proceed similarly, replacing $\tilde B_t(x)$ by the $\PSL(2,\R)$-valued
$\tilde B_t(x) R_{\psi(x)}$, and $\theta(t,x)$ by $\theta(t,x)+\phi(t,x)/2$:
we may thus also assume that $g=0$ in this case.  From now on, we
assume that $\alpha=\int \theta(0,0,x) \, d\mu(x)$.

Apply Lemma \ref {l.line adjust ext} with $T=[0,1]^2 \setminus
\{(0,0)\}$, $\varphi=\theta|(T \times X)$, $T^*=\emptyset$ and
$\varepsilon$ so small that $\varepsilon(0,s) < \int
\alpha-\theta(0,s,x) \, d\mu(x)$, $\varepsilon(s,0) < \int
\theta(s,0,x)-\alpha \, d\mu(x)$ and $\|\tilde B_t\|^2_\infty
\varepsilon(t) \to 0$ as $t \to (0,0)$.  Let $\tilde \varphi$ and
$\tilde \psi$ be as in the lemma. We now define $\hat
A_t(x)=\tilde B_t(f(x)) R_{2 \pi \tilde \varphi(t,x)} \tilde
B_t(x)^{-1}$ for $t \in [0,1]^2 \setminus \{(0,0)\}$ and $\hat
A_0=A$: this is a continuous family at $t=(0,0)$ since $\|\hat
A_t-\tilde A_t\|_\infty<2 \pi \varepsilon(t) \|\tilde
B_t\|_\infty^2$ for $t \in [0,1]^2 \setminus \{(0,0)\}$. Since
$\tilde \varphi(t,x)$ is cohomologous to $c(t) = \int \theta (t,x)
\, d\mu(x) = \rho(\tilde A_t)$, we can write $\hat A_t(x) = \hat
B_t(f(x)) R_{2 \pi c(t)} \hat B_t(x)^{-1}$ (where the $\psi$'s are
absorbed by $\hat B$).

Now choose some function $\eps_1(t) > 0$ that goes to zero as $t
\to (0,0)$. By Lemma~\ref{l.topology}, there exists a continuous
path $\gamma:(0,1] \to (0,1]^2$ whose range is contained in the
set of $t$'s for which $\| \hat A_t(x) - \hat B_t(f(x)) R_{2 \pi
\alpha} \hat B_t(x)^{-1} \| < \eps_1(t)$. (The latter set is open
and contains $K = c^{-1}(\alpha)$.) Thus, we may set $A_s(x) =
\hat B_{\gamma(s)}(f(x)) R_{2 \pi \alpha} \hat
B_{\gamma(s)}(x)^{-1}$. By construction, this family has the
desired properties.
\end{proof}

\begin{rem} \label{r.homotopic const}
It follows from the proof that if $\alpha = \rho_{f,A} \mod G'$,
then the maps $B_t$ can be taken in $C_0(X,\SL(2,\R))$ (i.e., homotopic to constant).
Conversely, if the conclusion of Theorem~\ref{t.reduce} holds with
$B_t \in C_0(X,\SL(2,\R))$, then it is straightforward to show that $\rho_{f,A} = \alpha \mod G'$.
\end{rem}

\begin{rem} \label{r.unlocked}
A cocycle can be accessed by a continuous path of unlocked cocycles with the
same rotation number if and only if it is unlocked itself.
Indeed, the ``only if'' part is an immediate consequence of the definitions.
On the other hand, notice that
if two cocycles in $C_0(X,\SL(2,\R))$ are conjugate via a map also in
$C_0(X,\SL(2,\R))$ then\footnote{In fact, it is not necessary that the conjugacy is homotopic to constant;
see Proposition~\ref{p.conjugacy inv}.}
they have the same type (locked / unlocked / semi-locked).
Thus, given an unlocked cocycle $A$,
applying Theorem~\ref{t.reduce} (and Remark~\ref{r.homotopic const}) with $\alpha=\rho_{f,A}$
we obtain a continuous path of unlocked cocycles with the same rotation number.\end{rem}

\begin{appendix}

\section{Connectedness Considerations}

Recall that $C_0(X, \SL(2,\R))$ indicates the subset of $C(X, \SL(2,\R))$
formed by maps that are homotopic to a constant.
The rotation number can be viewed as a map $C_0(X,\SL(2,\R)) \to \R / G'$.

\begin{prop}\label{p.const}
If two cocycles have the same rotation number {\rm (}mod~$G'${\rm )}, then
is it possible to connect them by a continuous path with constant  rotation number {\rm (}mod~$G'${\rm )}.
\end{prop}

\begin{proof}
Let $A'$, $A''\in C_0(X,\SL(2,\R))$ be two cocycles such that $\rho_{f,A'} = \rho_{f,A''} \mod G'$.
Fix once and for all a lift $F': X \times \R \to X \times \R$ such that
$e^{2 \pi i F_2'(x,t)}$ is a positive real multiple of $A'(x) \cdot e^{2 \pi i t}$.
This determines the fibered rotation number of $(f,A')$ as some $a \in \R$.
Moreover, for any contractible space $Y$ and any continuous map $y \in Y \mapsto A_y \in C_0(x,\SL(2,\R))$
such that $A_{y_0} = A'$ for some $y_0 \in Y$,
there exists a unique continuous function $\rho:Y \to \R$,
such that $\rho(y) = \rho_{f,A_y} \mod G'$ and
$\rho(y_0) = a$.
We also call $\rho$ a continuous determination of the rotation number.
(In fact, the previous use of this terminology is a particular case.)

Take any continuous path $t\in [0,1] \mapsto A_t \in C_0(X,\SL(2,\R))$ with
$A_0 = A'$ and $A_1 = A''$.
Consider the continuous determination of the rotation number $\rho:[0,1] \to \R$ with $\rho(0)=a$.
Then $\rho(1) - \rho(0) \in G'$, thus there exists $\phi \in C(X,\Z)$ such that
$\rho(1) - \rho(0) = \int \phi \, d\mu$.
Extend the path of cocycles to the interval $[0,2]$ by defining
$A_t(x) = R_{-2 \pi (t-1) \phi(x)} A''(x)$ for $t \in [1,2]$.
Considering the corresponding extension $\rho:[0,2]\to \R$,
we have 
$\rho(2) = \rho(0) = a$.

Now consider the two-parameter family
$(t,\theta) \in [0,2] \times \R \mapsto R_{\theta} A_t \in C_0(X , \SL(2,\R))$
and the associated continuous determination of the rotation number
$\rho : [0,2] \times \R \to \R$ such that $\rho(0,0) = a$, and thus $\rho(0,2) = a$.
Notice that $\rho(t,\theta)$ is non-decreasing as a function of $\theta$,
and $\rho(t,\theta + 2 \pi) = \rho(t,\theta) + 1$.

We claim that the function $f(t) = \inf \{\theta \in \R : \rho(t,\theta) = a\}$ is continuous.
Indeed, fix any sequence $t_n$ in $[0,2]$ converging to some $t$
and such that $\beta = \lim f(t_n)$ exists.
Let $\alpha = f(t)$. 
Then $\rho(t,\beta) = \lim \rho(t_n, \beta) \le a = \rho(t,\alpha)$.
It follows that $\beta \le \alpha$.
Moreover, if this inequality is strict, then $\rho(t, \theta)$ is constant equal to $a$ for $\theta$
in the interval $(\beta, \alpha)$ and thus by the first part of Proposition~\ref{p.lockprop},
$R_{(\alpha+\beta)/2} A_t$ is uniformly hyperbolic.
In particular $\rho(s, (\alpha+\beta)/2) = a$ for every $s$ sufficiently close to $t$.
This implies that $\rho(t,\beta) \le \lim \rho(t_n, \beta) \le \lim \rho (t_n , (\alpha+\beta)/2, t_n) = a$
and thus, $\beta \ge f(t)$ -- contradiction.

It now follows that the rotation number is constant mod~$G'$ along the
continuous path of cocycles $t\in [0,2] \mapsto R_{\min(0, f(t))} A_t$,
concluding the proof.
\end{proof}

\begin{prop}
If $A_0$, $A_1 \in C_0(X, \SL(2,\R)) \setminus \cUH$, then
we can join these cocycles by a continuous path $t \in [0,1] \mapsto A_t \in C(X, \SL(2,\R)) \setminus \cUH$.
Moreover, if $A_0$ and $A_1$ have the same rotation number {\rm (}mod~$G'${\rm )},
then we can take the path with constant rotation number {\rm (}mod~$G'${\rm )}.
\end{prop}

\begin{proof}
The first assertion is a trivial consequence of Proposition~\ref{p.extension2} with $p=1$.

To prove the second assertion, it suffices to show that for any cocycle
$A \in C_0(X, \SL(2,\R)) \setminus \cUH$ and any $\alpha = \rho_{f,A} \bmod{G}$,
there is a path joining $A$ and $R_{2\pi\alpha}$ along which the rotation number is constant.

Given such $A$ and $\alpha$, apply
Theorem~\ref{t.reduce} and find
continuous paths $A_{t}$, $t \in [0,1]$
and $B_{t}$, $t \in (0,1]$ in $C(X, \SL(2,\R))$ such that
$A_{0} = A$ and $A_{t}(x) = B_{t}(f(x)) R_{2\pi\alpha} B_{t}(x)^{-1}$.

Since $\SO(2,\R)$ is a deformation retract of $\SL(2,\R)$,
there exists a homotopy $h_t: \SL(2, \R) \to \SL(2,\R)$, $t \in [0,1]$
between the identity $h_0$ and a $\SO(2,\R)$-valued map $h_1$.

The rotation number is constant along the path
$A_{t+1}(x) = h_t(B_{1}(fx)) R_{2\pi\alpha} [h_t(B_{1}(x))]^{-1}$, $t\in [0,1]$
joining $A_{1}$ and the $\SO(2,\R)$-valued cocycle $A_2$.
The latter is homotopic to constant as a map $X \to \SL(2,\R)$
and thus (using the deformation retract) also as a map $X \to \SO(2,\R)$.
It follows that we can write $A_2(x) = R_{\theta_0(x)}$ where $\theta_0$ is real-valued.
Let $\theta_t$, $t\in [0,1]$ be a homotopy in $C(X,\R)$ between $\theta_0$ and constant $2\pi\alpha$
such that $\int \theta_t \, d\mu$ does not depend on $t$.
Define $A_{2+t}(x) = R_{\theta_t(x)}$ for $t\in [0,1]$.
Then $A_t$, $t\in [0,3]$ is a path joining $A$ and $R_{2\pi \alpha}$
along which the rotation number is constant.
\end{proof}

On the other hand, the set of uniformly hyperbolic cocycles with a given rotation number is
not necessarily connected:

\begin{prop}\label{p.disconnect}
There exists a strictly ergodic homeomorphism $f$ of a finite-dimensional compact space $X$
and there exist two cocycles $A_0$, $A_1 \in C_0(X, \SL(2,\R))$
with the same rotation number 
that lie in distinct path-connected components of $\cUH$.
\end{prop}

\begin{rem}\label{r.a4}
One could raise the question of whether it is possible to add
an extra parameter in Theorem~\ref{t.mainthm2} and obtain the
following statement:
``If $A_s$, $s\in [0,1]$ is a path in $C_0(X,\SL(2,\R))$
with $2\rho_{f, A_s}$ constant and in $G$, then
there is a continuous family of cocycles $A_{s,t}$, $(s,t)\in [0,1]^2$
such that $A_{s,0} = A_s$ and $A_{s,t} \in \cUH$ for $t>0$.''
However, it follows from Propositions~\ref{p.const} and \ref{p.disconnect}
that such a statement is \emph{false} in general.
(On the other hand, it is possible to add
an extra parameter in Theorem~\ref{t.bounded}:
we have done that already in Proposition~\ref{p.extension2}.)
\end{rem}

To prove Proposition~\ref{p.disconnect} we will need the following well-known lemma.
Lacking a suitable reference, we give a proof of it.

\begin{lemma} \label{l.ergodic}
Let $\alpha \in \R \setminus \Q$ and $\beta \in \R$.
For a generic continuous
$\phi:\R/\Z \to \R$ with $\int \phi(x) \, dx = \beta$, the map\footnote{Maps of such form
are called \emph{Anzai skew-products} after \cite{Anzai}.}
$f_{\alpha,\phi}:\R^2/\Z^2 \to \R^2/\Z^2$ given by $f_{\alpha, \phi}(x,y)=(x+\alpha,y+\phi(x))$
is strictly ergodic.\footnote {
It is well-know that ergodic Anzai skew-products are in fact strictly ergodic; see
\cite{Furstenberg}, but this will not be used below.}
\end{lemma}

\begin{proof}

Let $h_i:\R^2/\Z^2 \to \R$, $i \geq 1$ be a dense sequence in
$C(\R^2/\Z^2,\R)$.  Let $U_n$ be the set of all continuous $\phi$ with
average $\beta$ such that for every ergodic invariant measure $\mu$ for
$f_{\alpha,\phi}$, we have $|\int h_n \, d\mu-\int h_n \, d\Leb|<1/n$.

Clearly, if $\phi \in \bigcap U_n$, then $\phi$ is uniquely ergodic.  So we
just have to show that all $U_n$ are dense.  Let $\phi_0$ be any trigonometric
polynomial.  We wish to find $\phi$ close to $\phi_0$ that belongs to $U_n$.
Now, clearly $\phi_0(x)=\psi_0(x+\alpha)-\psi_0(x)+\beta$ for some trigonometric
polynomial.

For an arbitrary continuous
function $g:\R/\Z \to \R$, let $\mu_g$ be the probability
measure on $\R^2/\Z^2$ that projects to Lebesgue measure in the first factor and is
supported on the graph of $x \mapsto x+g(x)$.
Consider the sequence of functions $\delta_k:\R/\Z \to \R$, such that $\delta_k(x)$ is twice the distance between $k x$ and $\Z$.
It is readily seen that $\mu_{\delta_k}$ converges in
weak-$*$ sense to $\Leb$.  It follows that for every continuous function
$g$, $\mu_{g+\delta_k}$
converges to $\Leb$ (which equals the push-forward of $\Leb$ under $(x,y)
\mapsto (x,y+g(x))$.

In particular, $\mu_{\psi_0+\delta_k} \to \Leb$.  Let $p/q$ be a continued
fraction approximant of $\alpha$ and take $\psi=\psi_0+\delta_q$.  Let
$\phi(x)=\psi(x+\alpha)-\psi(x)$.  Then any invariant measure for
$f_{\alpha,\phi}$ is a vertical translate of $\mu_\psi$, that is, it is
of the form $(T_t)_* \mu_\psi$ where $T_t(x,t)=(x,y+t)$.  It follows that
all of them are close to $\Leb$, so if we take $q$ sufficiently large, we
will have $\phi \in U_n$.  Finally, one computes $\|\phi-\phi_0\|_\infty
\leq 2 |q \alpha-p|<1/q$, so taking $q$ large also ensures that $\phi$ is
close to $\phi_0$.
\end{proof}

\begin{proof}[Proof of Proposition~\ref{p.disconnect}]
By the previous lemma, there exists a
strictly ergodic homeomorphism $f=f_{\alpha, \phi}$ of the torus $\R^2/\Z^2$
with $\int \phi(x_1) \, dx_1 = \alpha \in \R \setminus \Q$.
Let  $\psi: \R/\Z \to \R$ be continuous and such that
$\psi(x_1+\alpha)-\psi(x_1)$ is uniformly close to $\phi(x_1) - \alpha$.
(This can be done easily using the remark after Lemma~3 in \cite{ABD}.)
Let $u : \R^2/\Z^2 \to \R/\Z$ be given by:
$$
u(x_1, x_2) = x_1 - x_2 + \psi(x_1).
$$
Thus
$u(f(x_1,x_2)) - u(x_1, x_2) = \psi(x_1+\alpha) - \psi(x_1) + \alpha - \phi(x_1)$,
which is close to $0$.
Notice that $u$ is not homotopic to constant.
Define two elements of $C(\R^2/\Z^2, \SL(2,\R))$ by
$$
A_0(x_1, x_2) = \begin{pmatrix} 2 & 0 \\ 0 & 1/2 \end{pmatrix} , \quad
A_1(x_1, x_2) = R_{u(f(x_1,x_2))} \begin{pmatrix} 2 & 0 \\ 0 & 1/2 \end{pmatrix} R_{-u(x_1, x_2)} .
$$
These cocycles are uniformly hyperbolic and have the same rotation number
(because $u \circ f - u$ lifts to a function of zero mean).
But no path $A_t$, $t\in [0,1]$  joining $A_0$ and $A_1$ can be wholly contained in $\cUH$,
because the respective expanding directions
$e^u(A_0)$, $e^u(A_1):\R^2 / \Z^2 \to \P^1$ are not homotopic.
\end{proof}

\section{Minimal Examples With Many Persistently Closed Gaps}\label{a.appb}

In \cite {Be1}, Question 4, the question of genericity of Cantor spectrum is
asked in a more general context than strictly ergodic dynamical systems.
The Gap Labelling Theorem holds, in fact, for any homeomorphism of a compact
metric space, and Bellissard asked whether the presence of
periodic orbits could be the only obstruction for the genericity of
Cantor spectrum (when it is allowed at all by the
Gap Labeling Theorem).
Though we have answered this question affirmatively
in our more restricted setting, the answer to the original question is in
fact negative, and as we will see even a slight departure from
unique ergodicity (keeping minimality) may already prevent Cantor spectrum.

We will consider the well-known Furstenberg examples
(see \cite[page~585]{Furstenberg} or \cite[\S~5.5]{Parry})
of minimal, but non-uniquely ergodic skew-products of the form $f(x,y)=(x+\alpha,y+\phi(x))$,
$\phi \in C(X=\R^2/\Z^2,\R)$, $\alpha \in \R \setminus \Q$.
For these examples,
$G=\Z \oplus \alpha \Z$ is independent of the choice of the
invariant measure.

Since $\calL$ is dense in $(0,1)$, the Gap Labelling Theorem does not
prevent Cantor spectrum for a Furstenberg example, and in fact there are
choices of the sampling function exhibiting Cantor spectrum.\footnote {
Indeed, the converse of the Gap Labelling Theorem holds for generic
potentials which depend only on the first coordinate: this case reduces to
the strictly ergodic dynamics $x \mapsto x+\alpha$ which is covered by Corollary
\ref {c.gencgl}.}
However, the next theorem shows that Cantor spectrum (and hence the converse
of the Gap Labelling Theorem) is not generic for such dynamics.

\begin{thm}\label{t.furstenberg}

Let $f$ be a Furstenberg example.
\begin{enumerate}
\item For an open and dense set of $v \in
C(X,\R)$, $\Sigma_{f,v}$ is not a Cantor set,
\item For a generic set of $v \in C(X,\R)$, $\interior \Sigma_{f,v}$ is dense
in $\Sigma_{f,v}$.\footnote {It is reasonable to expect that for generic $v$,
$\partial \Sigma_{f,v}$ is a Cantor set (thus $\Sigma_{f,v}$ is not
generically a finite union of intervals).}
\item The interior of $C_0(X,\SL(2,\R)) \setminus \cUH$ is dense in
$C_0(X,\SL(2,\R)) \setminus \cUH$.
\end{enumerate}

\end{thm}

\begin{proof}

We start with the third statement.  Notice that
if $A \in C_0(X,\SL(2,\R)) \cap \overline
\cUH$, the fibered
rotation number (viewed as an element of $\R/G'=\R/\Z$) is independent of
the invariant measure,\footnote {It is enough to check this for
$A \in C_0(X,\SL(2,\R)) \cap \cUH$: consider
$\rho(A)$ as a continuous function of the convex set of invariant measures,
which (since $A \in \cUH$) takes values in $\frac {1} {2} \Z+\frac
{\alpha} {2} \Z$.}.

If $A \in C_0(X,\SL(2,\R)) \cap C(X,\SO(2,\R))$, it has the form
$A(x)=R_{\phi(x)}$ with $\phi:X \to \R$ continuous, and the fibered rotation
number is independent of the invariant measure if and only if $\int_X \phi \,
d\mu$ is.  Thus for an open and dense set of $A \in C_0(X,\SL(2,\R)) \cap
C(X,\SO(2,\R))$, $A \notin \overline \cUH$.  By Theorem~\ref{t.addendum}, if $A$
is not uniformly hyperbolic, it can be approximated by a cocycle which is bounded,
and hence by one in the complement of $\cUH$.  This proves (c).

Using Lemma~\ref{l.projection}, (c) implies that for $v_0 \in C(X,\R)$ and $E \in
\Sigma_{f,v_0}$, since $A_{E,v_0} \not \in \cUH$, we can perturb $v_0$ to $v$ such that
$A_{E,v} \notin \overline{\cUH}$.  Then $E \in \interior \Sigma_{f,\tilde v}$ for every
$\tilde v$ in a neighborhood of $v$.  This gives (a).  A simple Baire
category argument allows us to conclude (b).
\end{proof}

\section{Uniform Hyperbolicity and the Rotation
Number}\label{a.uhrn}

In this section we discuss in more detail the relationship between
the behavior of the rotation number in a neighborhood of a cocycle
that is homotopic to a constant and properties of the cocycle.
Recall the definitions of a locked/semi-locked/unlocked cocycle, which were given in Section~\ref{s.2}.

\begin{prop}\label{p.lockprop}
Let $A\in C(X, \SL(2,\R))$ be homotopic to a constant, and fix a
continuous determination of the fibered rotation number
$\rho:\cV \to \R$ on a neighborhood of $A$.
Then:
\begin{enumerate}
\item The function $r_A : \theta \mapsto \rho(R_\theta A)$,
defined on an open interval around $0$, is locally constant at
$\theta = 0$ if and only if $A \in \cUH$. Consequently, the
cocycle $A$ is:
\begin{itemize}
\item locked if $R_A$ is locally constant at $\theta = 0$;
\item semi-locked if $\theta = 0$ is in the boundary of a plateau of
$r_A$;
\item unlocked if the value $\rho(A)$ does not correspond
to a plateau of $r_A$.
\end{itemize}

\item The cocycle $A$ is uniformly hyperbolic if and only if
it is locked.

\item If the cocycle $A$ is semi-locked, it can be accessed by
uniformly hyperbolic cocycles of the form $A_\theta = R_\theta A$,
where $\theta$ runs through a {\rm (}non-degenerate{\rm )} interval of the form
$[0,\eps]$ or $[-\eps,0]$.

\item If $A$ takes values in $\SO(2,\R)$, or if $2 \rho(A) \not\in G$, then $A$ is unlocked.
\end{enumerate}
\end{prop}

Before giving the proof of this proposition, let us first
establish the following lemma.

\begin{lemma} \label{l.wind}
For any $\epsilon>0$, there exist $c>0$ and $\lambda>1$ with the
following properties. Let $v$ be a non-zero vector in $\R^2$.
Assume $A_0$, $A_1$, \ldots, $A_{n-1}$ are matrices in $\SL(2,\R)$
such that
$$
\|A_{n-1} \cdots A_1 A_0 (v) \| < c \lambda^n \|v\| \, .
$$
Then for any $w \in \R^2 \setminus \{0\}$, there exists $ \theta
\in [-\epsilon, \epsilon]$ such that the vectors
$$
R_\theta A_{n-1} R_\theta A_{n-2} \cdots  R_\theta A_0 R_\theta
(v) \quad \text{and} \quad w
$$
are collinear.
\end{lemma}

A set $I \subset \P^1$ is called an interval if it is non-empty,
connected, and its complement contains more than one point. If $I
\subset \P^1$ is an open interval, there is a Riemannian metric
$\|\mathord{\cdot}\|_I$ on $I$, called the \emph{Hilbert metric},
which is uniquely characterized by the following properties:
\begin{itemize}
\item any invertible linear map $A:\R^2 \to \R^2$ induces an
isometry between $I$ and $A(I)$ with the corresponding Hilbert
metrics; \item if $I \subset \P^1$ is the projectivization of
$\{(1,y)\in \R^2 :  y \in (0,\infty)\}$, then the Hilbert metric
is given by $\frac{dy}{y}$.
\end{itemize}
If $J\Subset I$ are open intervals, then  the metric
$\|\mathord{\cdot}\|_I$ restricted to $J$ is smaller than
$\|\mathord{\cdot}\|_J$  times a factor $\tanh \big(\diam_{I}
(J)/4) < 1$ (Garrett Birkhoff's formula).

\begin{proof}[Proof of Lemma~\ref{l.wind}.]
We can assume that $\epsilon < \pi/2$, otherwise the conclusion is
trivial. Let $|\mathord{\cdot}|$ indicate the Euclidean (angle)
metric on $\P^1$ (so that the length of $\P^1$ is $\pi$). Fix
constants $0<C_1<C_2$ (depending only on $\epsilon$) such that for
any open interval $I$ whose Euclidean length $|I|$ is between
$\epsilon$ and $\pi-\epsilon$, we have:
\begin{alignat*}{2}
\|\mathbf{u}\|_I &\ge C_1 |\mathbf{u}| &\quad &\text{for every $\mathbf{u}$ tangent to $I$;}\\
\|\mathbf{u}\|_I &\le C_2 |\mathbf{u}| &\quad &\text{if
$\mathbf{u}$ is a tangent vector at the midpoint of $I$.}
\end{alignat*}
Also fix a constant $\tau>1$ such that if $J \Subset I$ are open
intervals with the same midpoint and $|J|+ \eps = |I| \le \pi -
\eps$, then
$$
\|\mathbf{u}\|_{J} \ge \tau \|\mathbf{u}\|_I \quad \text{for every $\mathbf{u}$ tangent to $J$.}\\
$$

Now take $A_0$, \ldots, $A_{n-1} \in \SL(2,\R)$ and a unit vector
$v \in \R^2$. Let $\bar v \in \P^1$ correspond to $v$. Assume that
\begin{equation}\label{e.no cover}
\big\{R_\theta A_{n-1} \cdots  A_1 R_\theta A_0 R_\theta (\bar{v})
: \theta \in [-\epsilon, \epsilon] \big\} \neq \P^1 \, ;
\end{equation}
we want to show that $\|A_{n-1} \cdots A_0 \|$ is exponentially
large.

Let us define some open intervals $I_0$, $I_1$, \ldots, $I_n
\subset \P^1$. Let $I_0$ be the open interval with midpoint
$\bar{v}$ and length $|I_0| = \epsilon$. Define inductively $I_k$
as the open interval with the same midpoint as $A_{k-1}(I_{k-1})$
and length $|I_k| = |A_{k-1}(I_{k-1})| + \epsilon$. It follows
from \eqref{e.no cover} and the fact that the matrices $A_j$ preserve
orientation that $I_k$ is indeed an interval, with endpoints
$R_{-\eps/2} A_{k-1} \cdots R_{-\eps/2} A_0 R_{-\eps/2} (\bar{v})$
and $R_{\eps/2} A_{k-1} \cdots R_{\eps/2} A_0 R_{\eps/2}
(\bar{v})$, and length $|I_k| \le \pi - \epsilon$. Therefore for
any vector $\mathbf{u}$ tangent to $I_{k-1}$, we have
$\|DA_{k-1} (\mathbf{u})\|_{I_k} \le    
\tau^{-1} \|\mathbf{u}\|_{I_{k-1}}$.

Let $\mathbf{w}$ be a non-zero tangent vector at $\bar{v}$, and
write $B= A_{n-1} \cdots A_0$.
Then 
$$
|DB(\mathbf{w})| \le C_1^{-1} \| DB(\mathbf{w}) \|_{I_n} \le
C_1^{-1} \tau^{-n} \| \mathbf{w} \|_{I_0} \le C_1^{-1} C_2
\tau^{-n} |\mathbf{w}| \, .
$$
Since $\| B(v) \|^{-2} = |DB(\mathbf{w})|/|\mathbf{w}|$,
\footnote{Indeed, the slice of the unit disk corresponding to a
small angle $|\mathbf{w}|$ around $v$ is a region of $\R^2$ with
area $\|v\|^2 |\mathbf{w}|$ that is mapped by $B$ to a region of
area $\|B(v)\|^2 |DB(\mathbf{w})| + \cO(|\mathbf{w}|^2)$; but $B$
is area-preserving.} we get that $\| B(v) \|$ is exponentially
large, as desired.
\end{proof}

\begin{proof}[Proof of Proposition~\ref{p.lockprop}.]
We begin by proving the first statement in part (a). Obviously,
$\rho$ is locally constant around $A$ if it is uniformly
hyperbolic. Let us prove the other direction. As usual, fix a lift
$F_2 : X \times \R \to \R$ of $A$. The nearby lift of $R_\theta A$
is $F_2 + \theta$, hence the rotation number is given by
$$
\rho(R_\theta A) = \lim_{n \to \infty} \frac{1}{n}\int\big[(F_2
+\theta)^n(x,t) - t\big] \, d\mu(x) \quad \text{($t$ arbitrary),}
$$
where $(F_2+\theta)^n(x,t)$ is the second coordinate of the $n$-th
iterate of $(x,t) \mapsto {(f(x), F_2(x,t) + \theta)}$.

Assume that $(f,A)$ is not uniformly hyperbolic. Fix $\eps>0$. We
will show that $\rho(R_\eps A) > \rho(R_{-\eps} A)$. Let $c>0$ and
$\lambda>1$ be given by Lemma~\ref{l.wind}, depending on~$\eps$.
There exist $n_0 \in \N$ and $x\in X$ such that $\|A^{n_0}(x)\| <
c \lambda^{n_0}$. Let $G$ indicate the (positive measure) set of
such points $x$. Then  Lemma~\ref{l.wind} implies that
$$
(F_2 + \eps)^{n_0}(x,t) - (F_2 - \eps)^{n_0}(x,t) > \tfrac{1}{2}
\quad \text{for all $x \in G$, $t\in \R$.}
$$
When $x\not\in G$, the left-hand side is positive. Since orbits
under $f^{n_0}$ visit $G$ with an average frequency $\mu(G)$, we
conclude that
$$
\rho(R_\eps A) - \rho(R_{-\eps} A) = \lim_{k \to \infty}
\frac{1}{n_0 k}\int\big[(F_2 +\eps)^{n_0 k}(x,t) - (F_2
-\eps)^{n_0 k}(x,t)\big] \, d\mu(x) \ge \frac{\mu(G)}{2n_0} \, ,
$$
as we wanted to show. This proves the first part of (a).

Let us now prove the remaining parts of (a). Suppose $A$ is such
that $r_A: \theta \mapsto \rho(R_\theta A)$ is constant on $[0,\eps]$.
By what we just proved, it follows that $R_{\eps/2} A \in \cUH$.
Take a neighborhood $V$ of $A$ such that $R_{\eps/2} B \in \cUH$
and $\rho(R_{\eps/2} B) = \rho(R_{\eps/2} A) = \rho(A)$
for every $B \in V$. Since $r_A$ is non-decreasing,
we get $\rho(B) \leq \rho(R_{\eps/2} B) =
\rho(A)$ and hence $\rho$ attains a local maximum at $A$.
Analogously, if $A$ is such that $\theta \mapsto \rho(R_\theta A)$ is constant on
$[-\eps,0]$, then $\rho$ attains a local minimum at $A$.
This establishes the remaining three implications in part~(a).

The parts (b), (c), and (d) follow readily from (a). (Too see the second part of
(d), note that if $A$ is not unlocked, then $\theta \mapsto
\rho(R_\theta A)$ is constant equal to $\rho_0$ on a interval $I$
of the form $[0, \eps]$ or $[-\eps, 0]$. Thus, $R_\theta A$ is
uniformly hyperbolic for all $\theta \in \interior I$ and hence
$2\rho(A) \in G$.)
\end{proof}

\begin{rem}
As is apparent from the proof, Proposition~\ref{p.lockprop}
holds in the more general setting where $f$ is a homeomorphism of
a compact space with a fully supported invariant probability
measure $\mu$ and the rotation numbers are taken with respect to
this measure (see \cite{H}, \S5.8.2).
\end{rem}

\begin{prop}\label{p.conjugacy inv}
Conjugate cocycles have the same type {\rm (}locked / semi-locked / unlocked{\rm )}.
\end{prop}

\begin{proof}
Let $B \in C(X,\SL(2,\R))$ be such that $x \mapsto B(f(x)) B(x)^{-1}$
is homotopic to a constant and let $\Phi:C_0(X,\SL(2,\R)) \to C_0(X,\SL(2,\R))$
be given by $\Phi(A)(x)=B(f(x)) A(x) B(x)^{-1}$.

Let $\cU$ be an open connected subset of $C_0(X,\SL(2,\R))$ such that
there exists a continuous choice of lifts $A \in \cU \mapsto F_{A,2} \in C(X \times \R,\R)$.
These lifts induce determinations $\rho:\cU \to \R$ and $\rho':\Phi(\cU) \to \R$
of the fibered rotation number.
We are going to show that
there exists $c \in \R$ such that
$\rho'(\Phi(A))=\rho(\Phi(A))+c$ for every $A\in \cU$.
The proposition follows immediately.

Let also $\Theta:X \times \R \to \R$ be a measurable
lift of $B$ that is continuous in
the second variable, but not necessarily in the first variable, such that
$\Theta(X \times \{0\})$ is bounded.
It follows that there exists a bounded measurable
function $\psi:\cU \times X \to \Z$ such that
$$
F_{\Phi(\tilde A), 2}(f(x),\Theta(x,t)) =
\Theta(f(x),F_{\tilde A, 2}(x,t)) + \psi(\tilde A,x).
$$
Thus $\rho'(\Phi(\tilde A))-\rho(A)=\int \psi(\tilde A,x) \, d\mu(x)$.
On the other hand, $\psi(\tilde A,x)$ is easily seen to be continuous
in the first variable, and since $\cU$ is connected, it depends only on the
second variable.
\end{proof}

\section{Proof of the Projection Lemma}\label{a.2}

Let us now prove Lemma~\ref{l.projection}. In the case where $\tr
A$ does not vanish identically, this result is contained in Lemma
9 of \cite {ABD}.  We will explain the necessary changes to deal
with the remaining case.

We wish to show that an arbitrary continuous $\SL(2,\R)$ perturbation
$\tilde A$ of the constant cocycle $A$ is conjugate to a Schr\"odinger
perturbation (with a conjugacy depending continuously on the perturbation,
and equal to $\id$ at $A$).  On the other hand, given any fixed compact
set of non-empty interior
$K \subset X$, Lemma 10 of \cite {ABD} states that under
the assumption of minimality, $\tilde A$ is conjugate to a perturbation
$\hat A$ such that $\hat A=A$ outside $K$.

So it is enough to describe how
to conjugate such a localized perturbation.  Since $X$ has at least four
points, we can choose $K$ admitting a compact neighborhood $K'$ such that
$f^i(K') \cap K'=\emptyset$ for $1 \leq i \leq 3$.  Notice that $A^4=\id$,
so $\hat A^4$ is close to $\id$.  Define $E=\sup_{x \in X}
\|\hat A^4(x)-\id\|^{1/2}$.

For each $x \in K'$, there are unique $E_i(x)
\in \R$, $1 \leq i \leq 4$, such that $E_3(x)=E \phi(x)$ and
$S_{E_4(x)} \cdots S_{E_1(x)}=\hat A^4(x)$, where for $t \in \R$, $S_t$ denotes the Schr\"odinger matrix $\begin{pmatrix} t & -1 \\ 1 & 0 \end{pmatrix}$.
Indeed, write $\hat A^4(x)=R_{-\pi/2} \hat A(x)=\id+\begin{pmatrix}
a(x)&b(x)\\c(x)&d(x) \end{pmatrix}$, and notice that $E_1(x)$, $E_2(x)$ and
$E_4(x)$ must be given by
$$
E_1(x)=\frac {b(x)-E \phi(x)} {1+d(x)}, \quad E_2(x)=-\frac {d(x)} {E \phi(x)},
\quad E_4(x)=\frac {c(x)+E_2(x)} {1-d(x)},
$$
with the expression for $E_2(x)$ being evaluated as $0$ whenever
$E \phi(x)=0$.  These expressions immediately imply that the maps $E_i:K' \to \R$ are
continuous and vanish in $\partial K'$.

We now define $\Phi(\hat A)$ as follows.  Outside
$\bigcup_{i=0}^3 f^i(K')$, we set $\Phi(\hat A)=\hat A=A$.
For $x \in K'$, we define $\Phi(\hat A)(f^i(x))=S_{E_{i+1}(x)}$,
$0 \leq i \leq 3$.
It is easy to see that $\Phi$ depends continuously on $\hat A$ in a
neighborhood of $A$.

By construction, $\Phi(\hat A)^4(x)=\hat A^4(x)$ for $x \in K$.  Define
$\Psi(\hat A)=\id$ outside $\bigcup_{i=0}^3 f^i(K')$,
$\Psi(\hat A)(f^i(x))=\Phi(\hat A)^i(x) \hat A^i(x)^{-1}$ for $x \in K'$ and
$1 \leq i \leq 3$.  Then $\Psi$ is a continuous function of $\hat A$
with $\Psi(\hat A)(f(x)) \cdot \hat A(x)
\cdot (\Psi(\hat A)(x))^{-1}=\Phi(\hat A)(x)$ for every $x \in X$.
This completes the proof of Lemma~\ref{l.projection}.

\begin{rem}

Like \cite[Lemma 9]{ABD}, Lemma~\ref{l.projection} is also valid
in smooth settings.  We only need to be careful in two steps.
First, we should choose $\phi$ to be a smooth function.
Second, we should replace the definition of $E=E(\hat A)$, which should be
``much larger than $d$'' for our argument to work, by something that takes
into account the derivatives of $d$.  In the case of $C^r$ cocycles
with finite $r$, we could take $E=\|d\|^{1/2}_{C^r}$, and it is
straightforward to check that both $\Phi$ and $\Psi$ are
$C^r$-valued continuous maps defined in a $C^r$-neighborhood of $A$.
In the case $r=\infty$, we can take
$$
E=-\frac {1} {\log \sup_{x \in K'} \|\hat A^4(x)-\id\|}
$$
whenever $\hat A \neq A$, and let $E=0$ when $\hat A=A$. With this
definition, $\Phi$ and $\Psi$ are $C^0$-valued continuous maps in
a $C^0$-neighborhood of $A$, but their restrictions to $C^{s+1}$
are $C^s$-valued continuous maps for every finite $s$ (using the
convexity inequality $\| \mathord{\cdot} \|_{C^s} \leq K_s (\| \mathord{\cdot}
\|_{C^0} \| \mathord{\cdot}\|_{C^{s+1}}^s)^{1/(s+1)}$; see \cite {N}), so
their restrictions to $C^\infty$ are $C^\infty$-valued continuous
maps.
\end{rem}

\section{Mixing Dynamics and Cantor Spectrum}\label{a.fayad}

In \cite {F}, Fayad proved that there exists a real-analytic
{\it time change} $f^t$ of an irrational flow on $\R^3/\Z^3$ which is mixing.
Thus $f^t$ is generated by a vector field of the form
$u \alpha$, where the coordinates
of $\alpha \in \R^3$ are rationally independent and
$u:\R^3/\Z^3 \to \R$ is a positive
real analytic function.
By general properties of time changes (see \cite[\S 5.1]{Parry}), to any
measure $\mu_0$ which is invariant by the flow generated by
$\alpha$, we can associate an $f^t$-invariant measure
which is absolutely continuous with respect to $\mu_0$ and has
density $1/u$, and all $f^t$-invariant measures are obtained in this way.
Since irrational flows are uniquely ergodic, the flow $f^t$ is uniquely
ergodic as well.  We will denote by $\mu$ the unique $f^t$-invariant
probability measure.

\begin{prop}

Let $f=f^1$ be the time one map of the flow $f^t$.  Then $f$ is mixing,
strictly ergodic, and the range $G$ of the Schwartzman asymptotic cycle of
$f$ is dense in $\R$.

\end{prop}

\begin{proof}

Obviously $f$ is mixing.  If $\mu'$ is any invariant probability
measure for $f$, then
$\int_0^1 f^t_* \mu' dt$ is an $f^t$-invariant probability measure, so it
coincides with $\mu$.  On the other hand, for every $\epsilon>0$,
$\mu'_\epsilon=\frac {1} {\epsilon} \int_0^\epsilon f^t_* \mu' dt$ is an
$f$-invariant probability measure absolutely continuous with respect to
$\mu$.  Obviously $\mu'_\epsilon \to \mu'$ as $\epsilon \to 0$,
and since $f$ is mixing, $f_*^n \mu'_\epsilon \to \mu$ as $n \to
\infty$ (in both cases the limits are taken in the weak-$*$ topology).
Since $\mu'_\epsilon$ is $f$-invariant, we conclude that
$\mu'_\epsilon=\mu$ for every $\epsilon>0$, and hence $\mu'=\mu$.  Thus
$f$ is uniquely ergodic.  Since $\supp \mu=\R^3/\Z^3$, unique ergodicity
implies strict ergodicity.

Given $k \in \Z^3$, let $\psi:\R^3/\Z^3 \to \R/\Z$ be given by
$\psi(x)=\langle x,k \rangle$, and let $\phi(x)=\int_0^1
u(f^t(x)) dt \langle \alpha,k \rangle$.  Then
$\phi(x)=\psi(f(x))-\psi(x) \bmod{1}$, $g=\int_{\R^3/\Z^3} \phi d\mu \in G$.
Recalling that $\mu$ is an $f^t$-invariant probability measure with
density proportional to $\frac {1} {u(x)}$, we get
\begin{multline*}
\int_{\R^3/\Z^3} \phi(x) d\mu(x)=
\int_0^1 \int_{\R^3/\Z^3} u(f^t(x)) d\mu(x) dt \langle \alpha,k\rangle=
\int_{\R^3/\Z^3} u(x) d\mu(x) \langle \alpha,k \rangle\\
=C^{-1} \langle \alpha,k \rangle,
\end{multline*}
where $C=\int_{\R^3/\Z^3} u(x)^{-1} dx>0$
is the normalization constant making $\mu$ a probability measure.
Since the coordinates of $\alpha$ are rationally independent, this shows
that $G$ contains a free Abelian subgroup of rank $3$, so $G$ is
dense in $\R$.
\end{proof}

It follows from the proposition and Corollary~\ref{c.gencgl}
that Schr\"odinger operators arising from mixing strictly ergodic dynamics
can have Cantor spectrum, thus answering negatively \cite[Problem~5]{Be0}.

\end{appendix}


\begin{thebibliography}{MMY}


\bibitem[An]{Anzai}
\textsc{H.~Anzai.}
Ergodic skew product transformations on the torus.
\textit{Osaka Math.\ J.}, 3 (1951), 83--99.

\bibitem[At]{A}
\textsc{K.~Athanassopoulos.}
Some aspects of the theory of asymptotic cycles.
\textit{Exposition.\ Math.}, 13 (1995), 321--336.

\bibitem[AB]{AB}
\textsc{A.~Avila, J.~Bochi.}
A uniform dichotomy for generic $\SL(2,\R)$ cocycles over a minimal base.
\textit{Bull.\ Soc.\ Math.\ France}, 135 (2007), 407--417.

\bibitem[ABD]{ABD}
\textsc{A.~Avila, J.~Bochi, D.~Damanik.} Cantor spectrum for
Schr\"odinger operators with potentials arising from generalized
skew-shifts. \textit{Duke Math.\ J.}, 146 (2009), 253--280.

\bibitem[AF]{AF}
\textsc{A.~Avila, G.~Forni.}
Weak mixing for interval exchange transformations and translation flows.
\textit{Ann.\ of Math.}, 165 (2007), 637--664.

\bibitem[AJ1]{AJ}
\textsc{A.~Avila, S.~Jitomirskaya.}
The Ten Martini Problem.
To appear in \textit{Ann.\ of Math.}

\bibitem[AJ2]{AJ2}
\textsc{A.~Avila, S.~Jitomirskaya.}
Almost localization and almost reducibility.
To appear in \textit{J.\ Eur.\ Math.\ Soc.}

\bibitem[AK]{AK}
\textsc{A.~Avila, R.~Krikorian.}
Reducibility or nonuniform hyperbolicity for quasiperiodic Schr\"odinger cocycles.
\textit{Ann.\ of Math.}, 164 (2006), 911--940.

\bibitem[AS1]{AS}
\textsc{J.~Avron, B.~Simon.}
Almost periodic Schr\"odinger operators. I.~Limit periodic potentials.
\textit{Comm.\ Math.\ Phys.}, 82 (1981), 101--120.

\bibitem[AS2]{AS2}
\textsc{J.~Avron, B.~Simon.}
Almost periodic Schr\"odinger operators. II.~The integrated density of states.
\textit{Duke Math.\ J.}, 50 (1983), 369--391.

\bibitem[BCL]{BCL}
\textsc{F.~B\'eguin, S.~Crovisier, F.~Le Roux.}
Construction of curious minimal uniquely ergodic homeomorphisms on manifolds: the Denjoy-Rees technique.
\textit{Ann.\ Sci.\ \'Ecole Norm.\ Sup.}, 40 (2007), 251--308.

\bibitem[Be1]{Be0}
\textsc{J.~Bellissard.} Schr\"odinger operators with almost
periodic potential: An overview. \textit{Mathematical Problems in
Theoretical Physics} (\textit{Berlin, 1981}), 356--363, Lecture
Notes in Phys., 153, Springer, Berlin-Heidelberg-New York, 1982.

\bibitem[Be2]{Be1}
\textsc{J.~Bellissard.} $K$-theory of $C\sp *$-algebras in solid
state physics. \textit{Statistical Mechanics and Field Theory:
Mathematical Aspects} (\textit{Groningen, 1985}), 99--156, Lecture
Notes in Phys., 257, Springer, Berlin, 1986.

\bibitem[Be3]{Be2}
\textsc{J.~Bellissard.}
Gap labelling theorems for Schr\"odinger operators.
\textit{From Number Theory to Physics} (\textit{Les Houches, 1989}), 538--630, Springer, Berlin, 1992.

\bibitem[Be4]{Be3}
\textsc{J.~Bellissard.} The noncommutative geometry of aperiodic
solids.  \textit{Geometric and Topological Methods for Quantum
Field Theory} (\textit{Villa de Leyva, 2001}),  86--156, World
Sci.\ Publ., River Edge, NJ, 2003.

\bibitem[BBG]{BBG}
\textsc{J.~Bellissard, R.~Benedetti, J.-M.~Gambaudo.} Spaces of
tilings, finite telescopic approximations and gap-labeling.
\textit{Comm.\ Math.\ Phys.}, 261 (2006), 1--41.

\bibitem[BLT]{BLT}
\textsc{J.~Bellissard, R.~Lima, D.~Testard.} Almost periodic
Schr\"odinger operators.  \textit{Mathematics + Physics. Vol. 1},
1--64, World Sci.\ Publishing, Singapore, 1985.

\bibitem[BSc]{BSc}
\textsc{J.~Bellissard, E.~Scoppola.} The density of states for
almost periodic Schr\"odinger operators and the frequency module:
a counterexample. \textit{Comm.\ Math.\ Phys.}, 85 (1982),
301--308.

\bibitem[BSi]{BS}
\textsc{J.~Bellissard, B.~ Simon.}
Cantor spectrum for the almost Mathieu equation.
\textit{J.\ Funct.\ Anal.}, 48 (1982), 408--419.

\bibitem[BO]{BO}
\textsc{M.-T.~Benameur, H.~Oyono-Oyono.}
Gap-labelling for
quasi-crystals (proving a conjecture by J. Bellissard).
\textit{Operator Algebras and Mathematical Physics}
(\textit{Constan\c ta, 2001}), 11--22, Theta, Bucharest, 2003.

\bibitem[BJ]{BJ}
\textsc{K.~Bjerkl\"ov, T.~J\"ager.} Rotation numbers for
quasiperiodically forced circle maps -- mode-locking vs.\  strict
monotonicity. \textit{J.\ Amer.\ Math.\ Soc.}, 22 (2009),
353--362.

\bibitem[Bo]{Bochi}
\textsc{J.~Bochi.}
Genericity of zero Lyapunov exponents.
\textit{Ergod.\ Th.\ Dynam.\ Sys.}, 22 (2002), 1667--1696.

\bibitem[BC]{BC}
\textsc{C.~Bonatti, S.~Crovisier.}
R\'ecurrence et g\'en\'ericit\'e.
\textit{Invent.\ Math.}, 158 (2004), 33--104.

\bibitem[Ca]{Cameron}
\textsc{R.~H.~Cameron.}
Linear differential equations with almost periodic coefficients.
\textit{Ann.\ of Math.}, 37 (1936), 29--42.

\bibitem[CEY]{CEY}
\textsc{M.~D.~Choi, G.~A.~Elliott, N.~Yui.}
Gauss polynomials and the rotation algebra.
\textit{Invent.\ Math.}, 99 (1990), 225--246.

\bibitem[Ch]{C}
\textsc{V.~A.~Chulaevski\u\i.}
An inverse spectral problem for limit-periodic Schr\"odinger operators.
\textit{Funktsional.\ Anal.\ i Prilozhen.}, 18 (1984), 63--66.

\bibitem[CS]{CS}
\textsc{V.~A.~Chulaevski\u\i, Ya~G.~Sina\u\i.}
Anderson localization for the $1$-D discrete Schr\"odinger operator with two-frequency potential.
\textit{Comm.\ Math.\ Phys.}, 125 (1989), 91--112.

\bibitem[CF]{CF}
\textsc{N.~D.~Cong, R.~Fabbri.} On the spectrum of the
one-dimensional Schr\"odinger operator. \textit{Discrete Contin.\
Dyn.\ Syst.}, 9 (2008), 541–-554.

\bibitem[Co]{Connes}
\textsc{A.~Connes.} An analogue of the Thom isomorphism for
crossed products of a $C\sp{*}$-algebra by an action of $\R$.
\textit{Adv.\ in Math.}, 39 (1981), 31--55.

\bibitem[DS]{DS}
\textsc{F.~Delyon, B.~Souillard.} The rotation number for finite
difference operators and its properties. \textit{Comm.\ Math.\
Phys.}, 89 (1983), 415--426.

\bibitem[E]{E}
\textsc{L.~H.~Eliasson.}
Floquet solutions for the $1$-dimensional quasi-periodic Schr\"odinger equation.
\textit{Comm.\ Math.\ Phys.}, 146 (1992), 447--482.

\bibitem[EJ]{EJ}
\textsc{R.~Ellis, R.~Johnson.}
Topological dynamics and linear differential systems.
\textit{J.\ Differential Equations}, 44 (1982), 21--39.

\bibitem[FJ]{FJ}
\textsc{R.~Fabbri, R.~Johnson.} Genericity of exponential
dichotomy for two-dimensional differential systems. \textit{Ann.\
Mat.\ Pura Appl.}, 178 (2000), 175--193.

\bibitem[FJP]{FJP}
\textsc{R.~Fabbri, R.~Johnson, R.~Pavani.} On the nature of the
spectrum of the quasi-periodic Schr\"odinger operator.
\textit{Nonlinear Anal.\ Real World Appl.}, 3 (2002), 37--59.

\bibitem[Fa]{F}
\textsc{B.~Fayad.}
Analytic mixing reparametrizations of irrational flows.
\textit{Ergodic Theory Dynam.\ Systems}, 22 (2002), 437--468.


\bibitem[Fu]{Furstenberg}
\textsc{H.~Furstenberg}
Strict ergodicity and transformation of the torus.
\textit{Amer.\ J.\ Math.}, 83 (1961), 573--601.

\bibitem[GS]{GS}
\textsc{M.~Goldstein, W.~Schlag.}
On resonances and the formation of gaps in the spectrum of quasi-periodic Schr\"odinger equations.
Preprint http://arxiv.org/abs/math/0511392.

\bibitem[HS]{HS}
\textsc{B.~Helffer, J.~Sj\"ostrand.}
Semiclassical analysis for Harper's equation. III.~Cantor structure of the spectrum.
\textit{M\'em.\ Soc.\ Math.\ France}, 39 (1989), 1--124.

\bibitem[H]{H}
\textsc{M.-R.~Herman.} Une m\'ethode pour minorer les exposants de
Lyapounov et quelques exemples montrant le caract\`ere local d'un
th\'eor\`eme d'Arnol'd et de Moser sur le tore de dimension $2$.
\textit{Comment.\ Math.\ Helv.}, 58 (1983), 453--502.

\bibitem[J]{J}
\textsc{R.~Johnson.} Exponential dichotomy, rotation number, and
linear differential operators with bounded coefficients.
\textit{J.\ Differential Equations}, 61 (1986), 54--78.

\bibitem[JM]{JM}
\textsc{R.~Johnson, J.~Moser.} The rotation number for almost
periodic potentials. \textit{Comm.\ Math.\ Phys.}, 84 (1982),
403--438.

\bibitem[KP]{KP}
\textsc{J.~Kaminker, I.~Putnam.} A proof of the gap labeling
conjecture. \textit{Michigan Math.\ J.}, 51 (2003), 537--546.

\bibitem[KZ]{KZ}
\textsc{J.~Kellendonk, I.~Zois.}
Rotation numbers, boundary forces
and gap labelling. \textit{J.\ Phys.\ A}, 38 (2005), 3937--3946.

\bibitem[K]{Knill}
\textsc{O.~Knill.}
The upper Lyapunov exponent of $\SL(2,\R)$ cocycles: discontinuity and the problem of positivity.
\textit{Lyapunov Exponents} (\textit{Oberwolfach, 1990}), 86--97, Lecture Notes in Math., 1486, Springer, Berlin, 1991.

\bibitem[L]{L}
\textsc{Y.~Last.}
Zero measure spectrum for the almost Mathieu operator.
\textit{Comm.\ Math.\ Phys.}, 164 (1994), 421--432.

\bibitem[MMY]{MMY}
\textsc{S.~Marmi, P.~Moussa, J.-C.~Yoccoz.}
The cohomological equation for Roth-type interval exchange maps.
\textit{J.\ Amer.\ Math.\ Soc.}, 18 (2005), 823--872.

\bibitem[M]{M}
\textsc{J.~Moser.}
An example of a Schr\"odinger equation with almost periodic potential and nowhere dense spectrum.
\textit{Comment.\ Math.\ Helv.}, 56 (1981), 198--224.

\bibitem[N]{N}
\textsc{L.~Nirenberg.} On elliptic partial differential equations.
\textit{Ann.\ Scuola Norm.\ Sup.\ Pisa}, 13 (1959), 115--162.

\bibitem[PT]{PT}
\textsc{L.~A.~Pastur, V.~A.~Tkachenko.}
Spectral theory of a class of one-dimensional Schr\"odinger operators with limit-periodic potentials.
\textit{Trans.\ Moscow Math.\ Soc.}, 1989 (1989), 115--166

\bibitem[PV]{PV}
\textsc{M.~Pimsner, D.~Voiculescu.} Exact sequences for $K$-groups
and Ext-groups of certain cross-product $C\sp{*}$-algebras.
\textit{J.\ Operator Theory}, 4 (1980), 93--118.

\bibitem[Pa]{Parry}
\textsc{W.~Parry.}
\textit{Topics in Ergodic Theory.}
Cambridge University Press, 1981.

\bibitem[Pu]{P}
\textsc{J.~Puig.}
Cantor spectrum for the almost Mathieu operator.
\textit{Comm.\ Math.\ Phys.}, 244 (2004), 297--309.

\bibitem[Sc]{S}
\textsc{S.~Schwartzman.} Asymptotic cycles. \textit{Ann.\ of
Math.}, 66 (1957), 270--284.

\bibitem[Si1]{Si1}
\textsc{B.~Simon.} On the genericity of nonvanishing instability
intervals in Hill's equation. \textit{Ann.\ Inst.\ H.\ Poincar\'e},
A24 (1976), 91--93.

\bibitem[Si2]{Si2}
\textsc{B.~Simon.} Almost periodic Schr\"odinger operators:
A review. \textit{Adv.\ Appl.\ Math.}, 3 (1982), 463--490.

\bibitem[X]{X}
\textsc{J.~Xia.} The Floquet theory and the state density of quasi-periodic Schr\"odinger operators.
\textit{J.\ Differential Equations}, 86 (1990), 243--259.

\bibitem[Y]{Yoccoz}
\textsc{J.-C.~Yoccoz.} Some questions and remarks about
$\SL(2,\mathbb{R})$ cocycles. In \textit{Modern Dynamical Systems
and Applications}, 447--458. Cambridge Univ.\ Press,
Cambridge, 2004.

\end{thebibliography}
\end{document}